\documentclass[journal]{IEEEtran}

\ifCLASSINFOpdf
\else
\fi
\usepackage{array,graphicx,cite,xcolor}
\usepackage[thinlines]{easytable}
\usepackage{graphicx}
\usepackage{tikz}
\usetikzlibrary{shapes,calc,snakes,spy}

\usepackage{amsmath,amsfonts,amssymb,mathtools}

\usepackage{amsthm,mathrsfs}
\usepackage{sidecap}
\usepackage{dblfloatfix}
\usepackage{subcaption}

\usepackage[printonlyused]{acronym}
\acrodef{hac}[HAC]{hybrid angle control}
\acrodef{coi}[COI]{center-of-inertia}
\acrodef{ib}[IB]{infinite bus}
\acrodef{sg}[SG]{synchronous generators}
\acrodef{wrt}[w.r.t.]{with respect to}
\acrodef{agas}[AGAS]{almost global asymptotic stability}
\acrodef{lhs}[LHS]{left-hand side}  
\acrodef{rhs}[RHS]{right-hand side}  
\acrodef{rocof}[RoCoF]{rate of change of frequency}

\usepackage{xpatch}
\makeatletter
\AtBeginDocument{\xpatchcmd{\@thm}{\thm@headpunct{.}}{\thm@headpunct{}}{}{}}
\makeatother

\usepackage{arydshln}
\usepackage{multicol,multirow}
\usepackage{pgfplots}
\usetikzlibrary[pgfplots.colormaps,matrix]

\usetikzlibrary{shapes,arrows}
\providecommand{\n}[1]{\lVert#1\rVert}
\providecommand{\ubar}[1]{\underline{#1}}
\providecommand{\tx}[1]{\text{\upshape{#1}}}	
\providecommand{\m}[1]{{\mathrm{#1}}}

\newcommand*\dt[0]{\frac{\text{d}}{\text{d}t}} 

\providecommand{\mc}[1]{\mathcal{#1}}                                                  		
\newtheorem{proposition}{Proposition}
\newtheorem{remark}{Remark}

\newtheorem{definition}{Definition}
\newtheorem{theorem}{Theorem}

\newtheorem{lemma}{Lemma}

\usepackage[font={small}]{caption}

\usepackage{todonotes}

\begin{document}
\title{Hybrid Angle Control and Almost Global Stability of Grid-Forming Power Converters}

\author{Ali~Tayyebi,~Adolfo~Anta,~and~Florian~Dörfler
	\thanks{A. Tayyebi (the corresponding author) is with the Austrian Institute of Technology, 1210 Vienna, Austria, and also with the Automatic Control Laboratory, ETH Zürich, 8092 Zürich, Switzerland, e-mail: ali.tayyebi-khameneh@ait.ac.at.}%
	\thanks{A. Anta is with the Austrian Institute of Technology, 1210 Vienna, Austria, e-mail: adolfo.anta@ait.ac.at.}%
	\thanks{F. Dörfler is with the Automatic Control Laboratory, ETH Zürich, 8092 Zürich, Switzerland, e-mail: dorfler@ethz.ch.}%
	\thanks{This work was partially funded by the independent research fund of the Austrian Institute for Technology, and ETH Zürich
	funds.}}
\maketitle
\begin{abstract}
This paper introduces a new grid-forming control for power converters, termed \ac{hac} that ensures the almost global closed-loop stability. 
\ac{hac} combines the recently proposed matching control with a novel nonlinear angle feedback reminiscent of (though not identical to) classic droop and dispatchable virtual oscillator controls. 
The synthesis of \ac{hac} is inspired by the complementary benefits of the dc-based matching and ac-based grid-forming controls as well as ideas from direct angle control and nonlinear damping assignment. 
The proposed \ac{hac} is applied to a high-fidelity nonlinear converter model that is connected to an infinite bus or a center-of-inertia dynamic grid models via a dynamic inductive line. 
We provide insightful parametric conditions for the existence, uniqueness, and global stability of the closed-loop equilibria. 
Unlike related stability certificates, our parametric conditions do not demand strong physical damping, on the contrary they can be met by appropriate choice of control parameters. 
Moreover, we consider the safety constraints of power converters and synthesize a new current-limiting control that is compatible with \ac{hac}. 
Last, we present a practical implementation of \ac{hac} and uncover its intrinsic droop behavior, derive a feedforward ac voltage and power control, and illustrate the behavior of the closed-loop system with publicly available numerical examples.%
\end{abstract}
\begin{IEEEkeywords}
grid-forming converter control, current-limiting control, power system stability, hybrid angle control.
\end{IEEEkeywords}
\IEEEpeerreviewmaketitle
\section{Introduction}
The Generation technology in power system has been drastically changing in recent years. The increasing replacement of bulk \ac{sg} with converter-interfaced generation is transforming the power system to a so-called \emph{low-inertia} system. The stability aftermath of this transition is highlighted by significant inertia reduction, fluctuating actuation (i.e., volatile generation), and the potential adverse interactions due to the presence of adjacent timescales \cite{MDHHV18,TGAKD20,MOVAH19,QQYYB:19,CTGAKF19,fang2018inertia,poolla_placement_2018}, among others. The \emph{grid-forming} control concept is envisioned to address the aforementioned stability challenges, whereby the converter features frequency and voltage regulation, black-start, and load-sharing capabilities \cite{TDKZH18}. 
	
Several grid-forming control techniques have been recently proposed. \emph{Droop control} mimics the speed droop of \ac{sg}, controls the modulation angle proportional to the active power imbalance, and is widely recognized as the baseline solution \cite{CDA93,SDB13}. As a natural extension of droop control, the emulation of \ac{sg} dynamics and control led to \emph{virtual synchronous machine} (VSM) strategies \cite{ZW11,d2013virtual}. The recently proposed \emph{matching} control exploits structural similarities of the converter and \ac{sg}; and matches their dynamics by controlling the modulation angle according to the dc voltage \cite{cvetkovic_modeling_2015,huang2017virtual,CGD17,AJD18,AF20}. Furthermore, \emph{virtual oscillator control} (VOC) mimics the dynamical behavior of Li\'enard-type oscillators and \emph{globally} synchronizes a converter-based network \cite{JMAFD:15,SDJD17}. Recently, \emph{dispatchable virtual oscillator control} (dVOC) is proposed that ensures \emph{almost global} synchronization of a homogeneous network of oscillator-controlled inverters (with simplified dynamics) to pre-specified set-points consistent with the power flow equations \cite{CGBF19,GCBD19} (also see \cite{yu2020comparative} for a comparative transient stability assessment of dVOC and droop control). 

A comparison of the aforementioned control strategies reveals complementary benefits;  see \cite[Rem. 2]{TGAKD20}: dc-based matching techniques are robust \ac{wrt} the load-induced over-currents and ac-based techniques (droop, VSM, and especially dVOC) have superior transient performance. Here we leverage these complementary benefits and design a \emph{hybrid angle control} (\ac{hac}) which combines matching control and a nonlinear angle feedback (reminiscent of, though not identical to, droop control and dVOC) and is inspired by ideas from direct angle control \cite{AF20} and sign-indefinite nonlinear damping assignment \cite{OVME02,sarras2012asymptotic}. Our proposed controller \emph{almost globally} stabilizes the closed-loop converter dynamics when connected via an inductive line to either an \ac{ib} or a dynamic \ac{coi} grid model. We provide insightful parametric conditions for the existence, uniqueness, and \emph{almost global stability} of closed-loop equilibria. Last but not least, we take into account the converter \emph{safety constraints}, design a new \emph{current-limiting} control, and investigate its stability in combination with \ac{hac}.

In contrast to most other related works, we consider a high-fidelity converter model including an explicit representation of energy source dynamics, the dc bus, LC filter, line dynamics, \ac{coi} grid dynamics, and the converter set-points. In comparison to related stability certificates \cite{BSEO17,CT14,AF20}, our stability conditions do not demand strong physical damping, but they can be met by appropriate choice of control gains. 

Moreover, our complementary choice of the angle-dependent terms in the \emph{Lyapunov / LaSalle} function and in the \ac{hac} formulation overcomes the analysis obstacles arising from lack of damping in angle state. Finally, we conclude this paper with some extensions, namely: a practical implementation and droop behavior of the \ac{hac} is described, a feedforward ac voltage and power control is discussed, and the behavior of the closed-loop system is explored with publicly available numerical examples \cite{T20}.
\begin{figure*}[!b]
\centering
{\includegraphics[trim=0.23cm 0.2cm 0.25cm 0.2cm, clip,width=\textwidth]{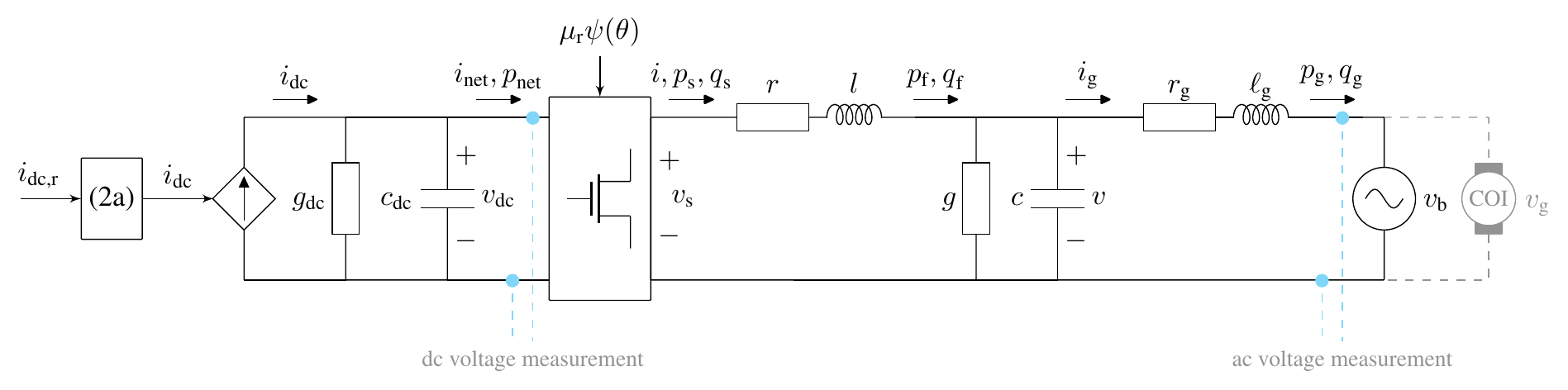}}
\caption{Schematic of the closed-loop system \eqref{eqs:sys}; see Figure \ref{fig:block diagram} for the control diagram defining $i_\tx{dc,r}$ and $\mu_\tx{r}\psi(\theta)$, and the Section \ref{subsec:COI} for the  description of the closed-loop system with the \ac{coi} grid model.\label{fig:SCIBsys}}
\end{figure*}

In addition, as an interesting technical contribution in its own right, this paper introduces an offbeat manifold space ---the boundary of a \emph{M\"obius strip}--- for studying the evolution of angle trajectories. A theoretical limitation of all (continuous control) systems with angles evolving on the circle is that they can at best achieve \ac{agas} due to the topological obstruction of the circle \cite{bhat2000topological}, which is a recurring theme in many of the aforementioned papers. Here we establish \ac{agas} of the angles on the boundary of a M\"obius strip, which results in \emph{global} asymptotic stability of the desired equilibrium when projected on the circle.

The remainder of this paper is structured as follows. Section \ref{sec:modeling} describes the model of a converter connected to an \ac{ib} and introduces the \ac{hac}. Section \ref{sec:closed-loop analysis} presents the closed-loop analysis and the main result of this work. Section \ref{sec:extensions} discusses two theoretical extensions: dynamic \ac{coi} grid model consideration and the design of a current-limiting control for grid-forming converters. Section \ref{sec:implementation} presents a practical \ac{hac} implementation, a complementary feedforward control, and \ac{hac}'s droop behavior. Next, we verify the performance of our controller via numerical examples in Section \ref{sec:simulation}. Last, a summary and outlook on future work are given in  Section \ref{sec:conclusion}.%
\section{Model Description}\label{sec:modeling}
\subsection{Preliminaries and Notation}
In this paper, $\mathbb{R}$ denotes the set of real numbers, $\mathbb{R}_{>0}$ denotes the set of strictly positive real numbers and $\mathbb{R}_{[a,b]}  \coloneqq  \{ x \in \mathbb{R} : a \leq x \leq b \}$. The unit circle i.e., one-dimensional torus is denoted by $\mathbb{S}^1$. For the column vectors $ x \in \mathbb{R}^n $ and $ x \in \mathbb{R}^m $, $ ( x , y )  \coloneqq  
\begin{bmatrix}
x^\top , y^\top 
\end{bmatrix}^\top
\in \mathbb{R}^{ n + m } $ denotes the stacked vector, and $ \m{I} $ is the two-dimensional identity matrix. The vector and matrix of zeros are respectively denoted by $0_n$ and $ {0}_{ n \times m } $. The block diagonal matrix is denoted by $ {\mathrm{blkdiag}}( \m{A}_1 , \ldots , \m{A}_n ) $.
Furthermore, $ \n{\cdot} $ denotes the Euclidean norm operator. Last, given $ \varphi \in \mathbb{S}^1 $ we define $ \psi ( \varphi )  \coloneqq  \big( \cos( \varphi )  , \sin( \varphi ) \big) $.

In this work, similar to \cite{TGAKD20,AJD18,ZW11,MOVAH19} among others, we consider symmetric three-phase electric circuitry assuming identical electrical parameters for all three phases and that all three-phase quantities $ z_\tx{ abc }  \coloneqq  ( z_\tx{a} , z_\tx{b} , z_\tx{c} ) \in \mathbb{R}^3 $ are balanced i.e., $ z_\tx{a} + z_\tx{b} + z_\tx{c} = 0$. Under the latter assumption, a three-phase quantity $ z_\tx{abc} $ is transformed to the stationary $ \alpha \beta $-frame via magnitude preserving Clarke transformation i.e., $z_{\alpha\beta}=\m{C}z_\tx{abc}$ (see Appendix \ref{app:Clarke} for details). Moreover, the image of $z_{ \alpha \beta }$ in direct-quadrature (dq) coordinates that rotate with constant frequency $ \omega_\tx{f} \in \mathbb{R}_{>0}$ and the angle $ \theta_\tx{f} = \omega_\tx{f} t + \theta_\tx{f}(0) \in \mathbb{S}^1$ is given by $ z = \m{R}( \theta_\tx{f} ) z_{\alpha\beta}$ with
\begin{equation*}
\m{R}( \theta_\tx{f} ) \coloneqq 
\begin{pmatrix}
\cos(\theta_\tx{f}) & -\sin(\theta_\tx{f})
\\
\sin(\theta_\tx{f}) & \cos(\theta_\tx{f})
\end{pmatrix}.
\end{equation*}
%
\subsection{Modeling the Connection of Converter and Stiff Grid}\label{subsec:CIB}
The \ac{ib} three-phase voltage is defined by
\begin{equation*}
v_\tx{b,abc}  \coloneqq  v_\tx{r} 
\big( \sin\left( \theta_\tx{b} \right) , \sin\left( \theta_\tx{b} - {2\pi} / {3} \right) , \sin\left( \theta_\tx{b} + {2\pi} / {3} \right) \big),
\end{equation*}
where $v_\tx{r} \in \mathbb{R}_{>0}$ is the nominal ac voltage magnitude, and 
\begin{equation}\label{eq: IB angle}
\theta_\tx{b}  \coloneqq  \omega_ 0 t + \theta_\tx{b}(0) \in \mathbb{S}^1
\end{equation}
is the \ac{ib} absolute angle with the nominal frequency $ \omega_{0} \in \mathbb{R}_{>0} $. 

We consider an average model of a three-phase two-level dc-ac converter \cite[Chap. 5]{YI10} and model the dc energy source by a first-order system that provides the input for a controlled dc current source. This is a reasonable coarse-grained model of the dc energy source e.g., see \cite[Sec. II.A]{TGAKD20}\cite[Sec. 6.4]{MO19}. The ac filter is modeled by an LC element. Moreover, the converter is interfaced to the \ac{ib} with an inductive line (that can also be seen as a low-voltage to medium voltage transformer model); Figure \ref{fig:SCIBsys} presents a schematic of the overall model.

The dynamical model of the converter-\ac{ib} system in $\alpha\beta$-frame is described by (see \cite[Sec. II]{AJD18}\cite[Chap. 5]{YI10} for a detailed derivation)
\begin{subequations}\label{eqs:converter}
	\begin{align}
	\tau_\tx{dc} \dot{i}_\tx{dc} & = i_\tx{dc,r} - i_\tx{dc},
	\label{eqs:converter0}
	\\
	c_\tx{dc} \dot{v}_\tx{dc} & = i_\tx{dc} - g_\tx{dc} v_\tx{dc} - m_{\alpha\beta}( \mu_\tx{r},\theta_\tx{c} )^\top  i_{\alpha\beta},
	\label{eqs:converter1}
	\\
	\ell \dot{i}_{\alpha\beta} & = v_\tx{dc} m_{\alpha\beta}( \mu_\tx{r},\theta_\tx{c} ) - r i_{\alpha\beta} - v_{\alpha\beta},
	\label{eqs:converter2}
	\\
	c \dot{v}_{\alpha\beta} & = i_{\alpha\beta} - g v_{\alpha\beta} - i_{\tx{g},{\alpha\beta}},
	\label{eqs:converter3}
	\\
	\ell_\tx{g} \dot{i}_{\tx{g},{\alpha\beta}} & = v_{\alpha\beta} - r_\tx{g} i_{\tx{g},{\alpha\beta}} - v_{\tx{b},{\alpha\beta}},
	\label{eqs:converter4}
	\end{align}
\end{subequations}
where $ \tau_\tx{dc}$ is the source time constant,  $ i_\tx{dc} \in \mathbb{R}$ is the dc source current, $ c_\tx{dc} $, $ v_\tx{dc} \in \mathbb{R} $, and $ g_\tx{dc} $ respectively denote the dc-link capacitance, voltage, and the dc conductance (that models the dc-side losses). 

Further, $ i_{\alpha\beta} $, $ v_{\alpha\beta} $, and $ i_{\tx{g},\alpha\beta} $ all take values in $\mathbb{R}^{2}$ and denote the current flowing through the filter inductance $\ell$, the voltage across the filter capacitance $c$, and the current through the line inductance $\ell_\tx{g}$. Lastly, $r$, $g$, and $r_\tx{g}$ model switching and conduction losses associated with the elements $\ell$, $c$, and $\ell_\tx{g}$, respectively. All parameters take positive and scalar values (due to the three-phase symmetry).

The modulation vector $m_{\alpha\beta}(\mu_\tx{r},\theta_\tx{c} ) \in \mathbb{R}_{ [ -1/2 , 1/2 ] } $ is $m_{\alpha\beta}( \mu_\tx{r},\theta_\tx{c} )  \coloneqq  \mu_\tx{r} \psi( \theta_\tx{c} )$ with reference magnitude $\mu_\tx{r} \in \mathbb{R}_{ [ 0 , 1/2 ] }$ and angle $\theta_\tx{c} \in \mathbb{S}^1$. In what follows, we will use the shorthand $m$ for $m_{\alpha\beta}( \theta_\tx{c},\mu_\tx{r} )$. The reference dc current in \eqref{eqs:converter0} is defined as 
\begin{equation}	\label{eq:i_dc}
i_\tx{dc,r} \coloneqq  i_\tx{r} - \kappa ( v_\tx{dc} - v_\tx{dc,r} ),
\end{equation}
where $i_\tx{r} \in \mathbb{R}$ denotes the open-loop dc current reference, $ \kappa \in \mathbb{R}_{>0} $ is the proportional gain of the dc voltage control, and $ v_\tx{dc,r} $ is the reference dc voltage. We remark that the forthcoming analysis also applies to the case with energy source being modeled as a stiff voltage source i.e., $\tau_\tx{dc} \to 0$ and $\kappa \to \infty$. 
\subsection{Hybrid Angle Control and Closed-Loop Dynamics}\label{subsec:SCIB}
We synthesize a new grid-forming strategy ---hybrid angle control (\ac{hac})---  by combining the dc-based matching control (see e.g., \cite[Eq. 25]{CGD17}) and a nonlinear angle feedback reminiscent of ---though not identical to--- droop control and dVOC (see e.g., \cite{SDB13,yu2020comparative} and Remark \ref{rem:variants} for details). Defining the converter relative angle \ac{wrt} the \ac{ib} as
\begin{equation}	\label{eq:relative angle}
\theta  \coloneqq  \theta_\tx{c} - \theta_\tx{b},
\end{equation}
 the \ac{hac} takes the form
\begin{equation}	\label{eq:omega_c}
\dot{\theta}_\tx{c}=\omega_\tx{c}  \coloneqq  \omega_0 + \eta ( v_\tx{dc} - v_\tx{dc,r} )  - \gamma \sin\left( \dfrac{ \theta - \theta_\tx{r} } { 2 } \right),
\end{equation}
where $\eta\in\mathbb{R}_{\geq0}, \gamma \in \mathbb{R}_{>0}$ are the control parameters and $\theta_\tx{r}$ denotes the reference relative angle (see Proposition \ref{prop:consistent references} on the implicit choice of $\theta_\tx{r}$ via deriving $\psi(\theta_\tx{r})$ from given set-points). Since the angle term in the \ac{rhs} of \eqref{eq:omega_c} is $4\pi$-periodic, the state ${\theta}_\tx{c}$ evolves on the set $ \mathbb{M}  \coloneqq  [ -2 \pi , 2 \pi ] $ with $\pm 2 \pi$ identified with each other. The terminology is due to $\mathbb M$ being the boundary of the M\"obius strip; see Remark \ref{rem:mobius} and Figure \ref{fig:Mobius}. This geometric curiosity will lead to profound insights later on.

Transforming the ac quantities in \eqref{eqs:converter2}-\eqref{eqs:converter4} to the dq-frame aligned with the \ac{ib} angle $\theta_\tx{b}$, we define the state vector
\begin{equation}\label{eq:states}
x  \coloneqq  ( \theta , i_\tx{dc}, v_\tx{dc} , {i} , {v} , {i}_\tx{g} ) \in \mathbb{X}  \coloneqq  \mathbb{M} \times \mathbb{R}^8
\end{equation}
and obtain the overall closed-loop dynamics \eqref{eqs:converter}-\eqref{eq:omega_c} as
\begin{subequations}	\label{eqs:sys}
	\begin{align}
	\dot{\theta} & = \omega_\tx{c} - \omega_0 = \eta ( v_\tx{dc} - v_\tx{dc,r} ) - \gamma \sin \big( { (\theta - \theta_\tx{r}) } / {2} \big),
	\label{eqs:sys1}
	\\
	\tau_\tx{dc} \dot{i}_\tx{dc} & = i_\tx{r} - \kappa ( v_\tx{dc} - v_\tx{dc,r} ) - i_\tx{dc},
	\label{eqs:sys1+}
	\\
	c_\tx{dc} \dot{v}_\tx{dc} & = i_\tx{dc} - g_\tx{dc} v_\tx{dc} - m^\top  i,
	\label{eqs:sys2}
	\\
	\ell \dot{{i}} & = v_\tx{dc} m - \m{Z} i - v,
	\label{eqs:sys3}
	\\
	c \dot{v} & = i - \m{Y} v - i_\tx{g},
	\label{eqs:sys4}
	\\
	\ell_\tx{g} \dot{i}_\tx{g} & = v - \m{Z}_\tx{g} i_\tx{g} - v_\tx{b},
	\label{eqs:sys5}
	\end{align} 
\end{subequations}
here $m=\mu_\tx{r}\psi(\theta)$, $\m{Z}  \coloneqq  r \m{I} - \ell \omega_0 \m{J}$ denotes the filter impedance matrix with $ \m{J}  \coloneqq  
\left( \begin{smallmatrix} 
0 & 1 \\ -1 & 0 
\end{smallmatrix} \right) $, $\m{Y}  \coloneqq  g\m{I} - c\omega_0 \m{J}$ is the shunt admittance matrix, 
$\m{Z}_\text{g}  \coloneqq  r_\text{g} \m{I} - \ell_\text{g}\omega_0\m{J}$ is the grid impedance matrix, 
and $ v_\tx{b}  \coloneqq  (v_\tx{r} ,0) $. 

For notational convenience we  respectively define the net dc current and power transferred to the converter ac-side as $i_\tx{net} \coloneqq  m^\top{i}$ and $p_\tx{net}:=v_\tx{dc}i_\tx{net}$. The ac active and reactive power injections at the switching node, the filter capacitance, and \ac{ib} nodes in Figure \ref{fig:SCIBsys} are respectively defined by $p_\tx{s}:=i^\top v_\tx{s}$ with $v_\tx{s} \coloneqq  v_\tx{dc}m$, $q_\tx{s}:=i^\top \m{J} v_\tx{s}$, $p_\tx{f}:=i^\top v$, $q_\tx{f}:=i^\top \m{J} v$ and $p_\tx{g}:=i_\tx{g}^\top v_\tx{b}$, $q_\tx{g}:=i_\tx{g}^\top \m{J} v_\tx{b}$ \cite[Sec. 4.6]{YI10}. Last, note that the \ac{rhs} of \eqref{eqs:sys} is continuously differentiable in $\mathbb{X}$ and the main nonlinearity aside \eqref{eqs:sys1} is represented by the modulated terms in \eqref{eqs:sys2} and \eqref{eqs:sys3} with their power-preserving structure (assuming lossless dc-ac conversion) i.e., $p_\tx{net}=p_\tx{s}$ \cite{AF20}. We close this section with three remarks on $\mathbb{M}$ and the \ac{hac}.
\begin{remark}\textup{(M\"obius strip)}	\label{rem:mobius}	\\
The angle term in \eqref{eq:omega_c} is $4\pi$-periodic and thus multi-valued on $\mathbb{S}^1$. Hence, we study the evolution of angle trajectories in $ \mathbb{M}$. One representation of $ \mathbb{M} $ is the compact boundary of M\"obius strip parametrized in $\mathbb R^{3}$ by $\sigma(w,\varphi)$ with coordinates
	\begin{subequations}
		\begin{align*}
		\sigma_1( w , \varphi ) &  \coloneqq  \big( \rho - w \cos\left( {\varphi} / {2} \right) \big) \cos\varphi,
		\\
		\sigma_2( w , \varphi ) &  \coloneqq  \big( \rho - w \cos\left( {\varphi} / {2} \right) \big) \sin\varphi,
		\\
		\sigma_3( w , \varphi ) &  \coloneqq  w \sin\big( {\varphi} / {2} \big),
		\end{align*}
	\end{subequations}
	where $ \rho \in \mathbb{R}_{>0} $ is the middle circle radius, $w$ denotes the half-width with $|w| \leq 1 / 2$, and $ \varphi \in \mathbb{S}^1$. Figure \ref{fig:Mobius} illustrates a parametrization of $\sigma(w,\varphi)$ such that $|\partial\sigma(w,\varphi)|=4\pi$ where $\partial\sigma(w,\varphi)$ denotes the boundary of m\"obius strip. 
\end{remark}
\begin{figure}[b!]
	\centering	
	{\includegraphics[trim=0.25cm 1.25cm 0.8cm 0.75cm, clip,width=0.75\columnwidth]{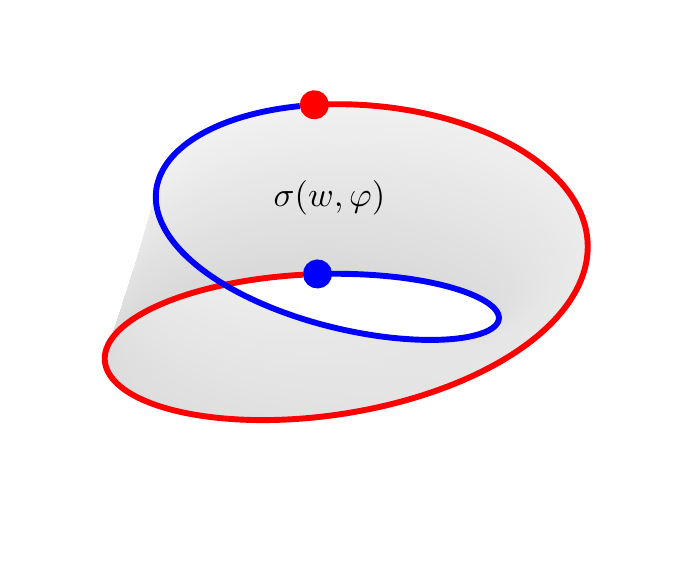}}
	\caption{The boundary of the M\"obius strip represents the angle space of \eqref{eqs:sys}. 
	 The arcs contained in the boundary segments colored in blue and red respectively represent 
	 the angles in $[-2\pi,0]$ and $[0,2\pi]$.
	}\label{fig:Mobius}
\end{figure}
\begin{remark}\textup{(\ac{hac} synthesis and properties)}\\
The \ac{hac} \eqref{eq:omega_c} is inspired by direct angle control \cite{AF20}, and blending the dc and ac dynamics in a grid-forming control design \cite[Rem. 2]{TGAKD20}. The angle feedback in \eqref{eq:omega_c} is an odd function and injects angle damping, i.e., it provides dissipation for the angle in \eqref{eqs:sys} in contrast to other control approaches, where the angle variable acts as a mere integrator. \ac{hac} offers two degrees of freedom ($\eta$ and $\gamma$) for optimal frequency tuning. Section \ref{sec:implementation} presents a practical implementation of \ac{hac}. 	The ratio of control gains $\eta$ and $\gamma$ not only trades off the \ac{hac} dependency on dc or ac dynamics, but also influences its power-frequency droop slope (see Proposition \ref{prop:droop slope}) and transient performance; see Subsection \ref{subsec:frequency performance} for a numerical example. 
\end{remark}
\begin{remark}\textup{{(\ac{hac} variants)}}	\label{rem:variants}	\\
It is noteworthy that with particular parameter choices in \eqref{eq:omega_c}, \ac{hac} recover several grid-forming controls. For instance, choosing $ \eta = 0 $ in \eqref{eq:omega_c} leads to a pure angle feedback control i.e., $\omega_\tx{c} = \omega_0 - \gamma \sin\big( (\theta - \theta_\tx{r}) / 2 \big)$ reminiscent of the droop control \cite[Sec. III-C]{TGAKD20}\cite{SDB13}. Indeed, the droop control is described by $\omega_\tx{c}: = \omega_{0} + d_{p-\omega} \big( p_\tx{r} - p\big)$, with the droop gain $d_{p-\omega}\in\mathbb{R}_{>0}$ (see Proposition \ref{prop:droop slope} for a definition), power reference $p_\tx{r}$, and $ p\coloneqq i_\tx{g}^\top v$ being the measured power that (with the assumptions in \cite{SDB13}) is proportional to $\sin( \theta )$. Likewise, \ac{hac} relates to dVOC, since around the equilibrium  dVOC dynamics (in polar coordinates) reduces to a form that resembles droop control \cite{yu2020comparative,seo2019dispatchable}. Further, the dc term with $\eta = \omega_0 / v_\tx{dc,r}$ recovers the standard matching control \cite{AJD18} combined with the angle term i.e., $\omega_\tx{c}=\eta v_\tx{dc} - \gamma \sin\big((\theta-\theta_\tx{r})/2\big)$. Last, if $\eta \neq \omega_0 / v_\tx{dc,r}$ the dc term in \eqref{eq:omega_c} is identical to the matching variant in \cite{CGD17}.
\end{remark}
\section{Closed-Loop Analysis}\label{sec:closed-loop analysis}
In what follows, we analyze the closed-loop system \eqref{eqs:sys} and provide suitable parametric conditions for the existence, uniqueness, and global stability of the closed-loop equilibria.
\subsection{Existence of Equilibria}
Provided that the dc voltage meets its reference in steady state, we will establish that the closed-loop system \eqref{eqs:sys} admits a unique equilibrium set containing two disjoint equilibria
\begin{equation}	\label{eq:EqSet}
\Omega^\star  
\coloneqq  
\big\{ 
x^\star_\tx{s}  \coloneqq  \left( \theta^\star_1 , {y}^{\star} \right) ,
x^\star_\tx{u}  \coloneqq  \left( \theta^\star_2 , {y}^{\star} \right) 
\big\},
\end{equation}
where $ \theta_1^\star   \coloneqq   \theta_\tx{r} $, $ \theta_2^\star   \coloneqq   \theta_\tx{r} + 2\pi $, and  $y^{\star}  \coloneqq  \left( {i}^{\star}_\tx{dc} , {v}^{\star}_\tx{dc} , {i}^{\star} , {v}^{\star} , {i}^{\star}_\text{g} \right)$ is the unique equilibrium of the states evolving in $\mathbb{R}^8$.
\begin{theorem}\textup{{(Existence of equilibria)}}\label{TH1}\\
The closed-loop system \eqref{eqs:sys} admits two equilibria if  there exist a consistent reference $i_\tx{r}=i^\star_\tx{dc}$ in \eqref{eqs:sys1+} such that $v_\tx{dc}^\star = v_\tx{dc,r} $. These disjoint equilibria only differ by their equilibrium angles being equal to $\theta_\tx{r}$ or $\theta_\tx{r} + 2\pi$, i.e., they are of the form \eqref{eq:EqSet}.
\end{theorem}
\begin{proof}
	We begin by setting the \ac{rhs} of \eqref{eqs:sys} to zero%
	\begin{subequations}\label{eqs:sys_ss}%
		\begin{align}%
		\eta(v_\tx{dc}^\star - v_\tx{dc,r}) - \gamma \sin\big((\theta^\star-\theta_\tx{r})/2\big)&=0,
		\label{eqs:sys_ss_theta}
		\\
		i_\tx{r}-\kappa(v^\star_\tx{dc}-v_\tx{dc,r})-i^\star_\tx{dc}&=0,
		\label{eqs:sys_ss_idc}
		\\
		i_\tx{dc}^\star-g_\tx{dc}v^\star_\tx{dc}-m(\theta^\star)^\top i^\star&=0,
		\label{eqs:sys_ss_vdc}
		\\
		v^\star_\tx{dc}m(\theta^\star)-\m{Z}i^\star-v^\star&=0_2,
		\label{eqs:sys_ss_i}
		\\
		i^\star-\m{Y}v^\star-i^\star_\tx{g}&=0_2,
		\label{eqs:sys_ss_v}
		\\
		v^\star-\m{Z}_\tx{g} i^\star_\tx{g}-v_\tx{b}&=0_2.
		\label{eqs:sys_ss_ig}    
		\end{align} 
	\end{subequations}
	If the condition of the theorem is met and thus $v^\star_\tx{dc}=v_\tx{dc,r}$, \eqref{eqs:sys_ss_idc} implies that $i^\star_\tx{dc}=i_\tx{r}$ and \eqref{eqs:sys_ss_theta} reduces to $\gamma \sin\big({(\theta^\star-\theta_\tx{r})}/{2}\big)=0$. Hence the angle equilibria are
	\begin{equation}\label{eq:theta_ss}
	\theta^\star_1=\theta_\tx{r}\quad \text{and}\quad \theta^\star_2=\theta_\tx{r}+2\pi.
	\end{equation}
	It remains to show \eqref{eqs:sys_ss_i}-\eqref{eqs:sys_ss_ig} admits a unique solution. Rearrange \eqref{eqs:sys_ss_i}-\eqref{eqs:sys_ss_ig} to $\m{A}(i^\star,v^\star,i_\tx{g}^\star)=b$ with
	\begin{align*}
	\m{A} \coloneqq \begin{pmatrix}
	-\m{Z} & -\m{I} & 0_{2\times 2} \\
	\m{I} & -\m{Y} & -\m{I}\\
	0_{2\times 2} &\m{I} & -\m{Z}_\text{g}
	\end{pmatrix}
	\end{align*}
	and $b \coloneqq \big(-v^\star_\tx{dc}m(\theta^\star),0_2,v_\tx{b}\big)$. It can be easily computed that symmetric part of $\m{A}$, that is, $\left({1}/{2}\right)\left(\m{A}+\m{A}^\top\right)\prec 0$.
	Hence, $\m{A}^{-1}$ exists and $(i^\star,v^\star,i_\tx{g}^\star) \coloneqq \m{A}^{-1}b$ is unique. 
\end{proof}
\begin{remark}{\textup{(Equilibria and the existence condition)}}\\
The following comments are in order: 

First, Theorem \ref{TH1} identifies two equilibria in \eqref{eq:EqSet} conditioned on a consistent reference $i_\tx{r}$ inducing $v^\star=v_\tx{dc,r}$. This condition can be enforced through appropriate feedforward control or feedback proportional-integral regulation, and it can be entirely omitted if $\eta = 0$ in \eqref{eq:omega_c}. 

Second, the condition is actually not restrictive. In fact, in the next section we will prove that the closed-loop system \eqref{eqs:sys} is \ac{agas} \ac{wrt}  $x^\star_\tx{s}$ under a mild parametric condition. Thus, if $x^\star_\tx{s}$ exists then no other relevant (i.e., stable) equilibria co-exist. 

Third, in case the set-points are inconsistent or if the system is subject to exogenous disturbances, then the closed-loop \eqref{eqs:sys} will undergo a so-called droop behavior similar to other grid-forming controls (e.g., \cite[Prop. 5]{AJD18}\cite[Prop. 4]{CGBF19}), i.e., a power imbalance will result in a frequency deviation; see Section~\ref{subsec:droop} for details.

Last, next to $x^\star_\tx{s}$ being \ac{agas} on $\mathbb{M} \times \mathbb{R}^8$, we will establish instability of $x^\star_\tx{u}$. In fact, when viewing the angle state not as element of $\mathbb{M}$ but more conventionally evolving on $\mathbb{S}^1$ (i.e., picture projecting Figure \ref{fig:Mobius} downwards to a circle), then the two equilibria $\{x^\star_\tx{s},x^\star_\tx{u}\}$ represent an identical point on $\mathbb{S}^1$. Thus, by working on $\mathbb{M}$ rather than $\mathbb{S}^{1}$ we by-passed the topological obstruction to continuous stabilization on $\mathbb{S}^1$ \cite{bhat2000topological}.%
\end{remark}
\subsection{Stability Analysis}
In the sequel, we establish the \ac{agas} of the closed-loop system \eqref{eqs:sys} \ac{wrt}  the equilibrium $x^\star_\tx{s}$ characterized in Theorem \ref{TH1}. This finding relies on the basis of intermediate results, namely, 1) global convergence of the trajectories to $\Omega^\star$ in \eqref{eq:EqSet}, 2) local asymptotic stability of $x^\star_\tx{s}$, and 3) instability of $x^\star_\tx{u}$. We begin by restating the definition of \ac{agas}  \cite[Def. 5]{CGBF19}. 
\begin{definition}\textup {(\ac{agas})}	\label{def:agas}	\\
An equilibrium of a dynamical system is almost globally asymptotically stable if it is asymptotically stable and for all initial conditions, except those contained in a Lebesgue zero-measure set, the trajectories converge to that equilibrium.
\end{definition}
Theorem \ref{TH2} below demonstrates the global attractivity of the equilibria \eqref{eq:EqSet} under a mild parametric condition. We want to highlight the innovative bounding scheme of the trigonometric error term $\psi(\theta)-\psi(\theta^\star)$ in the proof of Theorem \ref{TH2}, which is novel to our knowledge, and results in less restrictive parametric condition relative to the literature; see Remark~\ref{rem:feasibility}.
\begin{theorem}\textup{{(Global attractivity)}}	\label{TH2}	\\
Consider the closed-loop system \eqref{eqs:sys} and the equilibria $\Omega^\star$ characterized in Theorem \ref{TH1}. If the system parameters satisfy
	\begin{equation}	\label{eq:stability condition}
	\frac{ \eta }{ g_\tx{dc} }  + \frac{ \eta \left( \mu_\tx{r} \n{ i^{\star} } \right )^{ 2 } } { g_\tx{dc} } + \frac{ \eta ( \mu_\tx{r} v_\tx{dc}^{\star} ) ^ {2} }{ r } < \gamma, 
	\end{equation}
	then all trajectories of \eqref{eqs:sys} globally converge to $\Omega^\star$.
\end{theorem}
\begin{proof} 
	Define the error coordinates { $ \tilde{x} = \big( \tilde{\theta} , \tilde{i}_\tx{dc} , \tilde{v}_\tx{dc} , \tilde{i} , \tilde{v} , \tilde{i}_\tx{g} \big) $} \ac{wrt}  $ x^\star_\tx{s}$ -- with the equilibrium angle $\theta^\star_1$ -- in \eqref{eq:EqSet} as
	\begin{equation}	\label{eqs:ErrCoord}
	\tilde{x}  \coloneqq  
	( \theta - \theta^\star_1 , i_\tx{dc} - i_\tx{dc}^\star , v_\tx{dc} - v_\tx{dc}^\star , i - i^\star , v-v^\star ,i_\tx{g} - i_\tx{g}^\star ).
	\end{equation}
	The error dynamics associated with \eqref{eqs:sys} are described by
	\begin{subequations}	\label{eqs:ErrSys}
		\begin{align}
		\dot{ \tilde{\theta} } & = \eta \tilde{v}_\tx{dc} - \gamma \sin\big(\tilde{\theta}/2\big),
		\label{eqs:ErrSys1}
		\\
		\tau_\tx{dc}\dot{\tilde{i}}_\tx{dc} & = -\kappa\tilde{v}_\tx{dc} - \tilde{i}_\tx{dc},
		\label{eqs:ErrSys1+}
		\\
		c_\tx{dc} \dot{ \tilde{v} }_\tx{dc} & = \tilde{i}_\tx{dc} - g_\tx{dc} \tilde{v}_\tx{dc} - \mu_\tx{r} e_{\psi}^\top i^\star - m(\theta)^\top \tilde{i},
		\\
		\ell \dot{\tilde{i}} & = m(\theta) \tilde{v}_\tx{dc} +\mu_\tx{r} e_{\psi}v^\star_\tx{dc} - \m{Z} \tilde{i} - \tilde{v},
		\\
		c \dot{ \tilde{v} } & = \tilde{i} - \m{Y} \tilde{v} - \tilde{i}_\tx{g},
		\\
		\ell_\tx{g} \dot{\tilde{i}} & = \tilde{v} - \m{Z}_\tx{g} \tilde{i}_\tx{g},
		\end{align}
	\end{subequations}
	where $ e_{\psi}  \coloneqq  \psi( \theta ) - \psi( \theta^\star_\tx{1} ) $. Let { $\tilde{y}  \coloneqq  ( \tilde{i}_\tx{dc} , \tilde{v}_\tx{dc} , \tilde{i} , \tilde{v} , \tilde{i}_\tx{g} )$} and consider the composite parametric LaSalle/Lyapunov function
	\begin{equation}	\label{eq:LF}
	\mc{V}( \tilde{x} )  
	 \coloneqq  
	\mc{H}( \tilde{y}) + \lambda \mc{S}( \tilde{\theta} ) =
	\dfrac{1}{2}\big(\tilde{y}^\top \m{P} \tilde{y}\big) + 2 \lambda\left( 1 - \cos \dfrac {\tilde{\theta}}{2} \right)
	\end{equation}%
	where { $\m{P}  \coloneqq  {\mathrm{blkdiag} } \left( \tau_\tx{dc}/\kappa , c_\tx{dc} , \ell \m{I} , c\m{I} , \ell_\tx{g} \m{I} \right)$}, $\lambda \in \mathbb{R}_{>0}$, and $\mathcal{V}(\tilde{x})>0$ for all $\tilde{x}\neq 0_9$. See Figure \ref{fig:LF} for an illustration of $\mc{V}(\tilde{x})$. Evaluating $\dot{ \mc{V} }( \tilde{x} )$ along trajectories of \eqref{eqs:ErrSys} yields  
	\begin{align}
	\dot{ \mc{V} } ( \tilde{x} ) = 
	&  - g_\tx{dc} \tilde{v}^2_\tx{dc} - \dfrac{1}{\kappa}\tilde{i}^2_\tx{dc} - r \n{ \tilde{i} }^2 - g \n{ \tilde{v} }^2 - r_\tx{g} \n{ \tilde{i}_\tx{g} }^2
	\nonumber
	\\
	& +\lambda  \eta \tilde{v}_\tx{dc}\sin\big(\tilde{\theta}/2\big) -\lambda  \gamma \sin^2\big(\tilde{\theta}/2\big) 
	\nonumber
	\\
	& - \mu_\tx{r} \tilde{v}_\tx{dc} e_{\psi}^\top i^\star + \mu_\tx{r} v^\star_\tx{dc} \tilde{i}^\top e_{\psi} ,
	\label{eq:dotV}
	\end{align}
	where we exploited the skew symmetry of $ \m{J} $ in $\m{Z}$, $\m{Y}$, and $\m{Z}_\tx{g}$ i.e., $\tilde{{i}}^\top ( \ell \omega_0 \m{J} ) \tilde{i} = \tilde{v}^\top ( c \omega_0 \m{J} ) \tilde{v} = \tilde{i}_\tx{g}^\top ( \ell_\tx{g} \omega_0 \m{J} ) \tilde{i}_\tx{g} = 0. $ We apply the identity \eqref{id:binomial} to the terms in \eqref{eq:dotV} that depend on $e_{\psi}$:%
	\begin{subequations}	\label{eqs:BoundedTerms}
		\begin{align}
		- \mu_\tx{r} \tilde{v}_\tx{dc} e_{\psi}^\top i^\star & \leq \left( {\epsilon_1 \mu_\tx{r} \n{i^\star}}\right)^2{\tilde{v}_\tx{dc}}^2+\dfrac{1}{4\epsilon_1^2}\n{e_{\psi}}^2,\\
		\mu_\tx{r}{v}^\star_\tx{dc} \tilde{i}^\top e_{\psi} &\leq\epsilon_2^2\n{\tilde{i}}^2+\dfrac{(\mu_\tx{r} v_\tx{dc}^\star)^2}{4\epsilon_2^2}\n{e_{\psi}}^2.
		\end{align}
	\end{subequations}
	with $\epsilon_1,\epsilon_2\in\mathbb{R}_{>0}$. Next, by applying identities \eqref{id: sine half angle} and \eqref{id: cosine angle sum}, $\n{e_{\psi}}^2$ is expressed in terms of $\sin\big(\tilde{\theta}/2\big)$:
	\begin{align*}
	\n{e_{\psi}}^2
	&=\big(\cos(\theta)-\cos\left(\theta^\star_\tx{1}\right)\big)^2+\big(\sin(\theta)-\sin\left(\theta^\star_\tx{1}\right)\big)^2\\	
	&=2\big(1-\cos(\theta)\cos(\theta^\star_\tx{1})-\sin(\theta)\sin(\theta^\star_\tx{1})\big)\\
	&=2\big(1-\cos(\theta-\theta^\star_\tx{1})\big)=2\big(1-\cos\big(\tilde{\theta}\big)\big)=4\sin^2\big(\tilde{\theta}/2\big).
	\end{align*}
	Replace $\n{e_{\psi}}^2$ by $4\sin^2\big(\tilde{\theta}/2\big)$ in \eqref{eqs:BoundedTerms}, then $\dot{\mc{V}}(\tilde{x})$ in \eqref{eq:dotV} is upper-bounded by
	\begin{align}
	\dot{\mc{V}}(\tilde{x})\leq&~\lambda\eta\tilde{v}_\tx{dc}\sin\big(\tilde{\theta}/2\big)-\left(\lambda\gamma-\dfrac{1}{\epsilon_1^2}-\dfrac{(\mu_\tx{r} v_\tx{dc}^\star)^2}{\epsilon_2^2}\right) \sin^2\big(\tilde{\theta}/2\big)
	\nonumber
	\\
	&-\left(g_\tx{dc}-\left({\epsilon_1\mu_\tx{r} \n{i^\star}}\right)^2\right)\tilde{v}^2_\tx{dc}-\left(r-\epsilon_2^2\right)\n{\tilde{i}}^2
	\nonumber
	\\
	&-\dfrac{1}{\kappa}\tilde{i}^2_\tx{dc} - g\n{\tilde{v}}^2-r_\tx{g}\n{\tilde{i}_\tx{g}}^2=-\tilde{\zeta}^\top\m{Q}\tilde{\zeta},
	\label{eq:Vdot bound}
	\end{align}
	where $\tilde{\zeta} \coloneqq \big(\sin\big(\tilde{\theta}/2\big),\tilde{y}\big)$ and $\m{Q}   \coloneqq   
	\left(\begin{smallmatrix}	\m{Q}_{11} & 0_{3\times6}\\ 0_{6\times3} & \m{Q}_{22}	\end{smallmatrix}\right)$ with $\m{Q}_{11} \coloneqq $
	\begin{equation}\label{eq:N_11}
	\begin{pmatrix}
	\gamma\lambda-\dfrac{1}{\epsilon_1^2}-\dfrac{(\mu_\tx{r} v_\tx{dc}^\star)^2}{\epsilon_2^2}&0&-\dfrac{\eta\lambda}{2}\\ 
	0&\dfrac{1}{\kappa}&0\\
	-\dfrac{\eta\lambda}{2}&0&g_\tx{dc}-\left({\epsilon_1\mu_\tx{r} \n{i^\star}}\right)^2
	\end{pmatrix}
	\end{equation}
	and $\m{Q}_{22}   \coloneqq   \mathrm{blkdiag}\left(\left(r-\epsilon_2^2\right)\m{I},g\m{I},r_\tx{g}\m{I}\right)$. By standard Schur complement analysis, $\m{Q}\succ0$ iff
	\begin{subequations}
		\label{eq: Schur conditions}%
		\begin{align}
		\epsilon_{1}^{2} &< \dfrac{g_\tx{dc}}{(\mu_\tx{r} \|i^{\star}\|)^2} \coloneqq \alpha\quad\tx{and}\quad\epsilon_{2}^{2} < r,
		\label{eq: condition g_dc & r}
		\\
		\left(\dfrac{\lambda\eta}{2\sqrt{g_\tx{dc}}}\right)^2 &<\left(\gamma\lambda - \frac{1}{\epsilon_{1}^{2}} - \dfrac{(\mu_\tx{r} v_\tx{dc}^\star)^2}{\epsilon_2^2} \right)\left(1 - \dfrac{\epsilon_{1}^2}{\alpha}\right).
		\label{eq: condition gamma}%
		\end{align}%
	\end{subequations}%
These bounds can be optimized over the parameters $\epsilon_1,\epsilon_2$, and $\lambda$ to obtain the least conservative or most compact condition. 
	
To continue assume for now that $\eta>0$. The simple and favorable choice $\epsilon_1=\sqrt{\alpha/{2}}$, $\epsilon_{2}=\sqrt{{r}/{2}}$, and $\lambda=2/\eta$ yields that conditions \eqref{eq: Schur conditions} are satisfied and $\m{Q}\succ0$ iff the bound \eqref{eq:stability condition} is met. Accordingly, $\dot{\mc{V}}(\tilde{{x}})<0$ for all $\tilde{\zeta}\neq 0_9$. For $\eta = 0$, the off-diagonal elements of $\m{Q}_{11}$ in \eqref{eq:N_11} vanish. With the same choice of $\epsilon_1=\sqrt{\alpha/{2}}$ and $\epsilon_{2}=\sqrt{{r}/{2}}$, condition \eqref{eq: condition gamma} reduces to
	\begin{equation}\label{eq:stability condition eta=0}
	\frac{2\left(\mu_\tx{r} \n{i^{\star}}\right)^{2}}{\lambda g_\tx{dc}}  + \dfrac{2(\mu_\tx{r} v_\tx{dc}^\star)^2}{\lambda r} < \gamma\,.
	\end{equation}
	For any $\gamma>0$, \eqref{eq:stability condition eta=0} is met by a sufficiently large $\lambda>0$, which is consistent with condition \eqref{eq:stability condition} for $\eta=0$.
	
	Recall the  boundedness of $\tilde{\theta}$ in $\mathbb{M}$ and radial unboundedness of $\mc{H}(\tilde{{y}})$. Since $\dot{\mc{V}}(\tilde{{x}})\leq0$, for any $\tilde{x}(0)\in\mathbb{X}$, the set $\mathscr{L}_{\mc{V}(\tilde{x}(0))}=\{\tilde{x}\in\mathbb{X}:\mc{V}(\tilde{x})\leq \mc{V}\big(\tilde{x}(0)\big)\}$ is compact and forward invariant. Thus, by LaSalle's invariance principle \cite[Th. 4.4]{khalil_nonlinear_2002}, all trajectories of \eqref{eqs:ErrSys} converge to the largest invariant set in $\Omega=\{\tilde{{x}}\in\mathbb{X}:\dot{\mc{V}}(\tilde{{x}})=0\}$. 
	Since $\m{Q}\succ 0$, $\dot{\mc{V}}(\tilde{{x}})=0$ iff $\tilde{\zeta}=0_9$ which holds iff $\sin\big(\tilde{\theta}/2\big)=0$ and $\tilde{y}=0_8$ that means either $\tilde{\theta}=0$ or $\tilde{ \theta }=2\pi$ in $\Omega$ proving that $\Omega=\Omega^\star$.  
\end{proof}
\begin{remark}\textup{{(Feasibility, interpretation, and significance)}}\label{rem:feasibility}\\
Following comments are in order:	

First, condition \eqref{eq:stability condition} is met for sufficiently large $\gamma>0$ and it is possible to arbitrarily scale the \ac{lhs} terms via $\eta \geq 0$. Condition \eqref{eq:stability condition} implies that for small $g_\tx{dc}$ and $r$, a large ratio $\gamma/\eta$ is required. Also, for high $\n{i^\star}$, $v_\tx{dc}^\star$, and $\mu_\tx{r}$, the ratio $\gamma/\eta$ must be increased. Finally, for pure angle feedback, i.e., $\eta=0$ and $\gamma>0$ in  \eqref{eq:omega_c}, condition \eqref{eq:stability condition} is \underline{always} met  \underline{regardless} of the system equilibria and parameters.
	 
Second, condition \eqref{eq:stability condition} is significant because it does not demand a minimum physical damping unlike the conditions in \cite{AJD18,BSEO17,CT14,AF20}, e.g., see the $1/r$ proportionality in \eqref{eq:stability condition} or \cite[Th. 3]{AJD18}. In practice, this is met by considering a virtual impedance \cite{PD15} and / or a high-gain current control \cite{SGCD19}. In contrast, \ac{hac} only needs large enough ratio of gains  $\gamma/\eta$ to ensure stability. An interpretation is that with large enough angle damping, the inherent converter passivity is sufficient to stabilize the Euclidean states \cite{AJD18}. 
	
Third, condition \eqref{eq:stability condition} does not depend on the dc source time constant
$\tau_\tx{dc}$ and control gain $\kappa$. Thus, \eqref{eq:stability condition} also unveils the robustness of \ac{hac} \ac{wrt} delays in source actuation that suggests the compatibility of \ac{hac} with different energy sources on distinct timescales. Furthermore, $\kappa$ can still be freely chosen to optimize the dc voltage performance.	

Last, although \ac{hac} dominantly relies on ac dynamics (by recommendation for large $\gamma/\eta$), stabilization does not require the conventional timescale separation of angle, dc, and ac dynamics \cite{baros2020stability}.    	
\end{remark}	
Proposition \ref{prop:LAS} below reveals the local asymptotic stability of the equilibrium $x^\star_\tx{s}$ in \eqref{eq:EqSet} under the condition \eqref{eq:stability condition}. 
\begin{proposition}\textup{{(Local asymptotic stability)}}\label{prop:LAS}\\
	Consider the closed-loop system \eqref{eqs:sys} and assume that condition \eqref{eq:stability condition} holds. Then the equilibrium $x^\star_\tx{s}$ in \eqref{eq:EqSet} is locally asymptotically stable.   
\end{proposition}
\begin{proof}
	Consider the error dynamics \eqref{eqs:ErrSys} and the Lyapunov function \eqref{eq:LF} that satisfies $\mc{V}(0_9)=0$ and $\mc{V}(\tilde{x})>0$ for all $\tilde{x}\neq 0_9$. Furthermore, if  \eqref{eq:stability condition} holds, $\dot{\mc{V}}(\tilde{{x}})<0$ for all $\tilde{{x}}\neq0_9$ in a sufficiently small neighborhood of the origin. Consider a $c$-sublevel set of $\mc{V}(\tilde{x})$ i.e., $\mathscr{L}_{c} \coloneqq \left\{\tilde{x}\in\mathbb{X}:\mc{V}(\tilde{x})\leq c, c\in\mathbb{R}_{>0}\right\}$, which is forward invariant under the flow  \eqref{eqs:sys} since $\dot{\mc{V}}(\tilde{x})\leq 0$. Take $c$ to be sufficiently small such that the origin is the only equilibrium in $\mathscr{L}_{c}$ (recall that the equilibria in \eqref{eq:EqSet} are disjoint). Thus, by Lyapunov's direct method \cite[Th. 3.1]{khalil_nonlinear_2002} the origin is a locally asymptotically stable equilibrium of \eqref{eqs:ErrSys}.      
\end{proof}
Proposition \ref{prop: instability} reveals instability of $x^\star_\tx{u}$ in \eqref{eq:EqSet} and characterizes its region of attraction as a Lebesgue zero-measure set.

\begin{proposition}\textup{{(Unstable equilibrium point)}}\label{prop: instability}\\
	Consider the closed-loop system \eqref{eqs:sys} and assume that condition \eqref{eq:stability condition} holds. The equilibrium $x^\star_\tx{u}$ in \eqref{eq:EqSet} is unstable and its region of attraction has zero Lebesgue measure. 
\end{proposition}
The proof is provided in Appendix \ref{App:proofs}. Remark \ref{rem:saddle} reveals the topological peculiarity of the Lyapunov function \eqref{eq:LF} at $x^\star_\tx{u}$.
\begin{remark}\textup{{(Saddle point of the Lyapunov function)}}\label{rem:saddle}\\
	The unstable equilibrium $x^\star_\tx{u} = ( \theta^\star_2 , y^{\star} )$ is a min-max saddle point of the Lyapunov function $\mc{V}\left( x - x^\star_\tx{s}\right) = \mc{V}\left( \theta - \theta^\star_1 , y - y^\star\right)$ in \eqref{eq:LF} i.e., it can be shown that
	\begin{equation*}
	\mc{V}\left( \theta - \theta^\star_1 , 0_8 \right) 
	\leq 
	\mc{V}\left( \theta^\star_2 -\theta^\star_1 , 0_8 \right)
	\leq
	\mc{V}\left( \theta^\star_2 - \theta^\star_1 , y - y^\star\right),
	\end{equation*}
	for all $(\theta,y)$ in an open neighborhood of  $x^\star_\tx{u}$. This is due to the fact that the global minimum of $\mc{H}\left( y - y^\star \right)$ in $\mathbb{R}^8$ and the maximum of $\mc{S}\left( \theta - \theta_\tx{r} \right)$ in $\mathbb{M}$  coincide at $ x^\star_\tx{u} $;  see Figure \ref{fig:LF} for an abstract illustration assuming that $y,\theta\in\mathbb{R}$.
\end{remark}
\begin{figure}[b!]
\centering
{\includegraphics[trim=2mm 2mm 0 0.9mm,clip,width=\columnwidth]{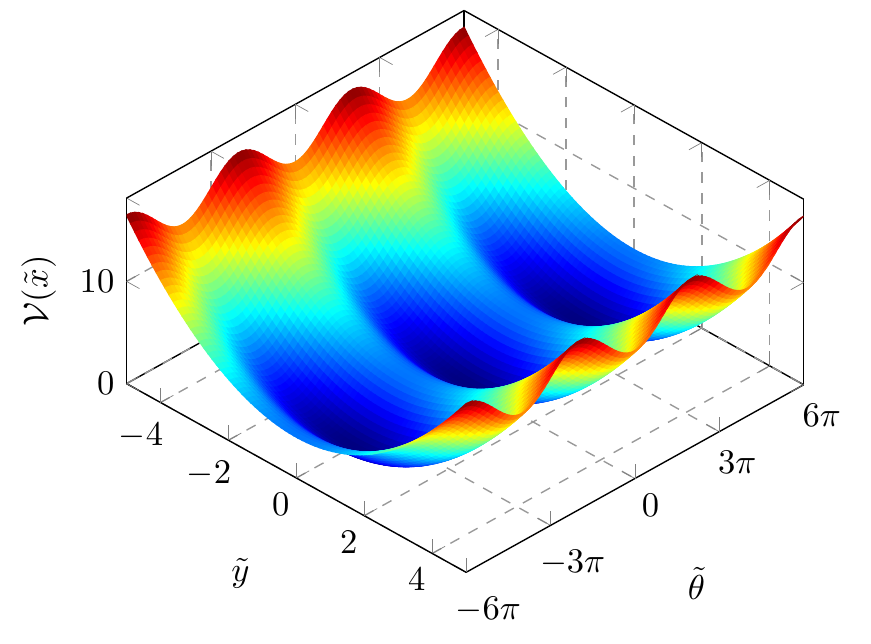}}
\caption{Illustration of the Lyapunov function \eqref{eq:LF} under the simplifying assumption that $ \tilde{y},\tilde{\theta}\in\mathbb{R}$, $\m{P}=1$, and $\lambda = 1$. \label{fig:LF}}
\end{figure}
Having established the intermediate results, Theorem \ref{TH:AGAS} presents our main result i.e., \ac{agas}  of system \eqref{eqs:sys} \ac{wrt}  $x^\star_\tx{s}$.
\begin{theorem}\textup{{(Main result: \ac{agas})}}	\label{TH:AGAS}\\
	Consider the closed-loop system \eqref{eqs:sys} and assume condition \eqref{eq:stability condition} holds. Then $x^\star_\tx{s}$ in \eqref{eq:EqSet} is almost globally asymptotically stable.
\end{theorem}
\section{Dynamic Grid and Current-Limiting Control}\label{sec:extensions}
In this section, we construct two extensions on the basis of the robust results in Section \ref{sec:closed-loop analysis}. First, we consider the connection of the converter to a dynamic grid represented by an \ac{coi} model and investigate closed-loop stability. Second, we account for the converter current constraint and design a  novel current-limiting control that is compatible with \ac{hac}.   
\subsection{Modeling the Connection of Converter and Dynamic Grid}\label{subsec:COI}
Under the slow coherency assumptions \cite[Chap. 2]{C13}\cite{RDB13}, an interconnected network of \ac{sg}s can be represented with an equivalent \ac{coi} model that relies on the aggregation of swing dynamics \cite[Sec. 6.10]{SP98}\cite[Sec. 3.2]{UBA14}. The angle and frequency dynamics of the \ac{coi} grid model are described by
\begin{subequations}\label{eqs:COI}
	\begin{align}
	\dot{\theta}_\tx{g}&=\omega,
	\label{eqs:COI1}
	\\
	J\dot{\omega}&=T_\tx{m}-D\omega-T_\tx{e},
	\label{eqs:COI2}
	\end{align}
\end{subequations}
where $J\in\mathbb{R}_{>0}$ is the moment of inertia and it is defined (in terms of the base power $S_\tx{r,g}$ and the inertia constant $H$) by $J:=2HS_\tx{r,g}/\omega_0^2$ \cite[Eq. 5.10]{machowski2020power}. Moreover, $T_\tx{m}\in\mathbb{R}$ denotes the mechanical torque, $D\in\mathbb{R}_{>0}$ denotes the aggregated damping and droop coefficient that models the aggregated governor action, and $T_\tx{e}\in\mathbb{R}$ is the electrical torque. 
Considering the structural similarity of \eqref{eqs:COI} and the full SG dynamics \cite{BSEO17}, we define the dynamic grid voltage as
\begin{equation}\label{eqs:COI3}
v_\tx{g,abc}  \coloneqq  b\omega 
\big( \sin\theta_\tx{g} , \sin\left( \theta_\tx{g} - {2\pi} / {3} \right) , \sin\left( \theta_\tx{g} + {2\pi} / {3} \right) \big),
\end{equation}
where $b\in\mathbb{R}_{>0}$ is constant \cite[Eq. 9]{BSEO17}. Thus, $T_\tx{e}$ in \eqref{eqs:COI2} can be written in terms of grid voltage and current as \cite[Eq. 10]{BSEO17}
\begin{equation}\label{eqs:COI4}
T_\tx{e}=\omega^{-1}i^\top _\tx{g,abc}v^{}_\tx{g,abc}. 
\end{equation}
Note that if $J\rightarrow\infty$ and $b$ is chosen such that $\n{v_\tx{g}}\rightarrow v_\tx{r}$ as $\omega\rightarrow\omega_0$, \eqref{eqs:COI}-\eqref{eqs:COI4} recovers the \ac{ib} grid model with constant frequency and voltage magnitude (see Section \ref{subsec:CIB}). 

Consider a converter controlled by \ac{hac} and connected to the \ac{coi} grid model via an inductive line; see Figure \ref{fig:SCIBsys}. Combining \eqref{eqs:converter}, \eqref{eq:i_dc}, \eqref{eq:omega_c}, and \eqref{eqs:COI}-\eqref{eqs:COI4} the overall closed-loop dynamics in a dq-frame aligned with $\theta_\tx{g}$ is
\begin{subequations}\label{eqs:sys COI}
	\begin{align}
	\dot{\theta}&=\omega_0+\eta(v_\tx{dc}-v_\tx{dc,r})-\gamma\sin\big((\theta-\theta_\tx{r})/2\big)-\omega,
	\label{eqs:sys COI1}
	\\
	\tau_\tx{dc}\dot{i}_\tx{dc}&=i_\tx{r}-\kappa(v_\tx{dc}-v_\tx{dc,r})-i_\tx{dc},
	\label{eqs:sys COI2}
	\\
	c_\tx{dc} \dot{v}_\tx{dc} & = i_\tx{dc} - g_\tx{dc} v_\tx{dc} - m^\top  i,
	\label{eqs:sys COI3}
	\\
	\ell \dot{i}&=v_\tx{dc}m-\m{Z}(\omega)i-v,
	\label{eqs:sys COI4}
	\\
	c \dot{v}&=i-\m{Y}(\omega)v-i_\tx{g},
	\label{eqs:sys COI4+}
	\\
	\ell_\tx{g}\dot{i}_\tx{g}&=v-\m{Z}_\tx{g}(\omega)i_\tx{g}-b\omega e_1,
	\label{eqs:sys COI5}
	\\
	J\dot{\omega}&=T_\tx{m}-D\omega+b e_1^\top i_\tx{g},
	\label{eqs:sys COI6}
	\end{align}
\end{subequations} 
where $\theta \coloneqq \theta_\tx{c}-\theta_\tx{g}$ denotes the converter relative angle \ac{wrt}  the \ac{coi} and $e_1 \coloneqq (1,0)$. Note the impedance and admittance matrices in \eqref{eqs:sys COI4}-\eqref{eqs:sys COI5} are frequency-dependent (cf. to constant matrices in \eqref{eqs:sys3}-\eqref{eqs:sys5}).
\subsection{Equilibria Characterization and Closed-Loop Stability}
To begin with, define the augmented closed-loop state vector
\begin{equation}\label{eq:statesSG}
\ubar{x}  \coloneqq  ( \theta , i_\tx{dc}, v_\tx{dc} , \omega , {i} , {v} , {i}_\tx{g} ) \in \ubar{\mathbb{X}}  \coloneqq  \mathbb{M} \times \mathbb{R}^9.
\end{equation}
Assume that the reference dc current and mechanical torque in \eqref{eqs:sys COI2} and \eqref{eqs:sys COI6} are set such that the equilibrium dc voltage and frequency coincide with $v_\tx{dc,r}$ and $\omega_0$. Following the same procedure as in the proof Theorem \ref{TH1}, dynamical system \eqref{eqs:sys COI} admits two equilibria with the same structure as $\Omega^\star$ in \eqref{eq:EqSet} i.e.,
\begin{equation}	\label{eq:EqSetCOI}
\ubar{\Omega}^\star 
 \coloneqq  
\big\{ 
\ubar{x}^\star_\tx{s}  \coloneqq  (\theta^\star_1 , \ubar{y}^{\star}) ,
\ubar{x}^\star_\tx{u}  \coloneqq  (\theta^\star_2 , \ubar{y}^{\star}) 
\big\},
\end{equation}
with $ \theta_1^\star   \coloneqq   \theta_\tx{r} $, $ \theta_2^\star   \coloneqq   \theta_\tx{r} + 2\pi $, and  $\ubar{y}^{\star}  \coloneqq  ( {i}^{\star}_\tx{dc} , {v}^{\star}_\tx{dc}, \omega^\star , {i}^{\star} , {v}^{\star} , {i}^{\star}_\text{g} )$ denoting the unique equilibrium associated with states evolving in $\mathbb{R}^9$. Theorem \ref{TH4} delivers the same result as Theorem \ref{TH2} i.e., global stability of $\ubar{\Omega}^\star$ under the flow defined by \eqref{eqs:sys COI}.
 \begin{theorem}\textup{{(Global stability with dynamic grid model) } }\label{TH4}	\\
 	Consider the closed-loop system \eqref{eqs:sys COI} and the equilibria $\ubar{\Omega}^\star$ defined in \eqref{eq:EqSetCOI}. If 
 	the system and control parameters satisfy
 	\begin{subequations}\label{eqs:stability conditions COI}
 		\begin{align}
 		&D>{D}_{{\tx{min}}} \coloneqq \dfrac{\left(\ell\n{i^\star}\right)^2}{r}+\dfrac{\left(c\n{v^\star}\right)^2}{g}+\dfrac{\left(\ell_\tx{g}\n{i_\tx{g}^\star}\right)^2}{r_\tx{g}},
 		\label{eqs:stability conditions COI1}
 		\\
 		&\gamma>\frac{ \eta }{g_\tx{dc} }  + \frac{ \eta \left( \mu_\tx{r} \n{ i^{\star} } \right )^{ 2 } } {g_\tx{dc} } + \frac{ \eta ( \mu_\tx{r} v_\tx{dc}^{\star} ) ^ {2} }{ r }+\dfrac{1}{2(D-{D}_{{\tx{min}}})},
 		\label{eqs:stability conditions COI2}
 		\end{align}
 	\end{subequations}
 	then all trajectories of \eqref{eqs:sys COI} globally converge to $\ubar{\Omega}^\star$.
 \end{theorem}
 \begin{proof}
Define the error coordinates as $\tilde{\ubar{x}} \coloneqq  \ubar{x} - \ubar{x}^\star_\tx{s}$. The error dynamics associated with \eqref{eqs:sys COI} are described by 
\begin{subequations}\label{eqs:ErrSysSG}
	\begin{align*}
	\dot{\tilde{\theta}} & = \eta \tilde{v}_\tx{dc} - \gamma \sin\big( \tilde{\theta} / 2 \big) - \tilde{\omega},
	\\
	\tau_\tx{dc} \dot{\tilde{i}}_\tx{dc} & = -\kappa \tilde{v}_\tx{dc} - \tilde{i}_\tx{dc},
	\\
	c_\tx{dc} \dot{ \tilde{v} }_\tx{dc} & = \tilde{i}_\tx{dc} - g_\tx{dc} \tilde{v}_\tx{dc} - \mu_\tx{r} e_{\psi}^\top i^\star - m(\theta)^\top \tilde{i}
	\\	
	\ell \dot{\tilde{i}} & = \tilde{v}_\tx{dc} m\big(\tilde{\theta}\big) + \mu_\tx{r} v_\tx{dc}^\star e_{\psi} - r \tilde{i} - \ell \omega \m{J} \tilde{i} - \ell \tilde{\omega} \m{J} i^\star - \tilde{v},
	\\
	c \dot{ \tilde{v} } & = \tilde{i} - g \tilde{v} - c \omega \m{J} \tilde{v} - c \tilde{\omega} \m{J} v^\star - \tilde{i}_\tx{g},
	\\
	\ell_\tx{g} \dot{\tilde{i}}_\tx{g} & = \tilde{v} - r_\tx{g} \tilde{i}_\tx{g} - \ell_\tx{g} \omega \m{J} \tilde{i}_\tx{g} - \ell_\tx{g} \tilde{\omega} \m{J} i^\star_\tx{g} - b e_1 \tilde{\omega},
	\\
	J \dot{\tilde{\omega}} & = - D \tilde{\omega} + b e^\top_1 \tilde{i}_\tx{g}.
	\end{align*}
\end{subequations}
where $ e_{\psi}  \coloneqq  \psi( \theta) - \psi( \theta^\star_\tx{1} ) $. Define $\tilde{\ubar{y}} \coloneqq (\tilde{i}_\tx{dc},\tilde{v}_\tx{dc},\tilde{\omega},\tilde{i},\tilde{v},\tilde{i}_\tx{g})\in\mathbb{R}^9$, and consider the following Lyapunov function
\begin{equation*}
\ubar{\mathcal{V}}(\tilde{\ubar{x}}) \coloneqq 
\dfrac{1}{2}\big(\tilde{\ubar{y}}^\top \m{\ubar{P}} \tilde{\ubar{y}}\big) + 2 \ubar{\lambda}\left( 1 - \cos \dfrac {\tilde{\theta}}{2} \right)
\end{equation*}
where $\m{\ubar{P}} \coloneqq \mathrm{blkdiag}(\tau_\tx{dc}/\kappa,c_\tx{dc},J,\ell\m{I},c\m{I},\ell_\tx{g}\m{I})$ and $\ubar{\lambda}\in\mathbb{R}_{>0}$. Evaluating $\dot{\ubar{\mathcal{V}}}(\tilde{x})$ along the error trajectories yields
\begin{align}\label{eq:VdotCOI}
\dot{ \ubar{\mathcal{V}} }( \tilde{x} ) =
& - \dfrac {1} {\kappa} \tilde{i}_\tx{dc}^2 - g_\tx{dc} \tilde{v}_\tx{dc}^2 - D \tilde{\omega}^2 - r \n{\tilde{i}}^2 - g\n{\tilde{v}}^2 - r_\tx{g} \n{\tilde{i}_\tx{g}}^2
\nonumber
\\
& - \ubar{\lambda} \gamma \sin^2\big( \tilde{\theta} / 2 \big) + \ubar{\lambda} \eta \tilde{v}_\tx{dc} \sin\big( \tilde{\theta} / 2 \big) + \ubar{\lambda} \tilde{\omega} \sin\big( \tilde{\theta} / 2 \big)
\nonumber
\\
& - \tilde{i}^\top ( \ell\m{J} i^\star ) \tilde{\omega} - \tilde{v}^\top ( c \m{J} v^\star ) \tilde{\omega} - \tilde{i}^\top_\tx{g} ( \ell_\tx{g} \m{J} i^\star_\tx{g} ) \tilde{\omega}
\nonumber
\\
& - \mu_\tx{r} e_{\psi}^\top i^\star \tilde{v}_\tx{dc} + \mu_\tx{r} v^\star_\tx{dc} \tilde{i}^\top e_{\psi}.
\end{align}
From the proof of Theorem \ref{TH2} recall that $\n{e_{\psi}}^2=4\sin^2\big(\tilde{\theta}/2\big)$ and apply \eqref{id:binomial} to the cross-terms in \eqref{eq:VdotCOI} that depend on $v^\star_\tx{dc}$, $i^\star$, $v^\star$, and $i_\tx{g}^\star$. Then $\dot{\ubar{\mathcal{V}}}(\tilde{\ubar{x}})$ is upper-bounded by 
\begin{align}\label{eq:VdotBoundSG}
\dot{ \ubar{\mathcal{V}} }( \tilde{\ubar{x}} ) \leq
& - \dfrac {1} {\kappa} \tilde{i}_\tx{dc}^2 - \left( g_\tx{dc} - \left( \ubar{\epsilon}_1 \mu_\tx{r} \n{i^\star} \right)^2 \right) \tilde{v}_\tx{dc}^2 - ( D - \ubar{\alpha} ) \tilde{\omega}^2 
\nonumber
\\
& - \big( r - \ubar{\epsilon}_2^2 - \ubar{\epsilon}_3^2 \big) \n{\tilde{i}}^2 - \big( g - \ubar{\epsilon}_4^2 \big) \n{\tilde{v}}^2 - \big( r_\tx{g} - \ubar{\epsilon}_5^2 \big) \n{\tilde{i}_\tx{g}}^2
\nonumber
\\
& - \left( \gamma \ubar{\lambda} - \dfrac {1} {\ubar{\epsilon}_1^2} - \dfrac {( \mu_\tx{r} v_\tx{dc}^\star )^2} {\ubar{\epsilon}_2^2} \right) \sin^2\big( \tilde{\theta} / 2 \big)
\nonumber
\\
& - \ubar{\lambda} \tilde{\omega} \sin\big( \tilde{\theta} / 2 \big) + \ubar{\lambda} \eta \tilde{v}_\tx{dc} \sin\big( \tilde{\theta} / 2 \big)
\end{align}
where $\ubar{\epsilon}_j\in\mathbb{R}_{>0}$ for $j=1,...,5$ and 
\begin{equation*}
\ubar{\alpha} \coloneqq \left( \dfrac { \ell \n{i^\star} } { 2 \ubar{\epsilon}_3 } \right)^2 + \left( \dfrac { c \n{v^\star} } { 2 \ubar{\epsilon}_4 } \right)^2 + \left( \dfrac { \ell_\tx{g} \n{i_\tx{g}^\star} } { 2 \ubar{\epsilon}_5 } \right)^2.
\end{equation*} 
Defining $\ubar{\zeta} \coloneqq \big( \sin\big( \tilde{\theta} / 2 \big) , \tilde{\ubar{y}} \big)$ the bound in \eqref{eq:VdotBoundSG} takes the quadratic form i.e., 
$\dot{ \ubar{\mathcal{V}} }(\tilde{x})\leq-\ubar{\zeta}^\top\m{\ubar{Q}}\,\ubar{\zeta},$
where $\m{\ubar{Q}}=\mathrm{blkdiag}\big(\m{\ubar{Q}}_{11},\m{\ubar{Q}}_{22}\big)$ with $\m{\ubar{Q}}_{11} \coloneqq $
\begin{equation*}
\begin{pmatrix}
\gamma\ubar{\lambda}-\dfrac{1}{\ubar{\epsilon}_1^2}-\dfrac{(\mu_\tx{r} v_\tx{dc}^\star)^2}{\ubar{\epsilon}_2^2}&0&-\dfrac{\eta\ubar{\lambda}}{2}&\dfrac{\ubar{\lambda}}{2}\\ 
0&\dfrac{1}{\kappa}&0&0\\
-\dfrac{\eta\ubar{\lambda}}{2}&0&g_\tx{dc}-\left({\ubar{\epsilon}_1\mu_\tx{r} \n{i^\star}}\right)^2&0\\
\dfrac{\ubar{\lambda}}{2}&0&0&D-\ubar{\alpha}
\end{pmatrix},
\end{equation*}
and $\m{\ubar{Q}}_{22} \coloneqq \mathrm{blkdiag}\big((r-\ubar{\epsilon}_2^2-\ubar{\epsilon}_3^2)\m{I},(g-\ubar{\epsilon}_4^2)\m{I},(r_\tx{g}-\ubar{\epsilon}_5^2)\m{I}\big).$
Choosing the free parameters as $\ubar{\lambda}=2/\eta$, $\ubar{\epsilon}_1=\sqrt{g_\tx{dc}}/\big(\sqrt{2}\mu_\tx{r} \n{i^\star}\big)$, $\ubar{\epsilon}_2=\sqrt{r/2}$, $\ubar{\epsilon}_3=\sqrt{r}/2$, $\ubar{\epsilon}_4=\sqrt{g}/2$ and $\ubar{\epsilon}_5=\sqrt{r_\tx{g}}/2$ renders $\m{\ubar{Q}}_{22}\succ0$. Under this favorable choice of parameters, $\m{\ubar{Q}}_\tx{11}\succ 0$ if and only if \eqref{eqs:stability conditions COI} is satisfied. Thus, $\m{\ubar{Q}}\succ 0$ and $\dot{\ubar{\mathcal{V}}}(\tilde{\ubar{x}})<0$ for all $\ubar{\zeta}\neq 0_{10}$. Following the LaSalle's invariance argument in the proof of Theorem \ref{TH2}, it is straightforward to show that the trajectories of \eqref{eqs:ErrSysSG} globally converge to the largest invariant set contained in $\ubar{\Omega} \coloneqq \big\{\tilde{\ubar{x}}\in\mathbb{\ubar{X}}:\dot{\ubar{\mathcal{V}}}(\tilde{\ubar{x}})=0\big\}$ and  $\ubar{\Omega}=\ubar{\Omega}^\star$.
\end{proof} 
\begin{remark}\textup{{(Extended stability conditions)} }\label{rem:COI}\\
Condition \eqref{eqs:stability conditions COI1} is met if the \ac{coi} model is sufficiently damped (see \cite{CT14,BSEO17,OVME02} for discussions on the critical damping requirement) and with large enough $\gamma$ condition \eqref{eqs:stability conditions COI2} is satisfied (see Remark \ref{rem:feasibility}). The conditions in \eqref{eqs:stability conditions COI}  \underline{almost} decouple into the previous stability condition (cf. \eqref{eqs:stability conditions COI2} and \eqref{eq:stability condition}) and the damping requirement \eqref{eqs:stability conditions COI1} (that can be perceived as characterization of the grid types to which the converter can be interfaced). For $D \gg {D}_{{\tx{min}}}$ conditions \eqref{eqs:stability conditions COI} reduce to \eqref{eq:stability condition}. An interpretation of the trade-off between $\gamma$ and $D$ is that with large $D$ the timescale of \ac{coi} model approaches that of the converter. Thus, synchronization demands less angle damping due to more coherent frequency dynamics. 
\end{remark}
The proof of Theorem \ref{TH4} extends the proof of Theorem \ref{TH2}. By following analogous arguments, it is possible to extend Propositions \ref{prop:LAS} and \ref{prop: instability} and Theorem \ref{TH:AGAS} to \eqref{eqs:sys COI} under condition \eqref{eqs:stability conditions COI}. For brevity of presentation, we omit the straightforward albeit lengthy calculations.
\subsection{Compatible Current-Limiting Control Synthesis}
In practice, power converters have tight state constraints for safety: in particular, their filter current magnitude needs to be constrained to a prescribed limit for protecting the semiconductor switches. 
There are ad hoc current-limiting techniques (without theoretical certificates e.g., see \cite{sadeghkhani2016current}) for \emph{grid-following} converters (see \cite{MDHHV18} for a definition). 
The design of current limitation strategies for grid-forming (i.e., voltage source) converters is an active research topic \cite{TGAKD20,GD19,TWDF20,PD15}. To this date, a universally satisfactory solution that safeguards the converter against various contingencies (e.g., load-induced over-current and balanced / unbalanced faults) is not known. 
In what follows, we propose a new current-limiting control and highlight its compatibility with the \ac{hac}. 

To begin with, by viewing the current dynamics in isolation, we derive the magnitude dynamics associated with \eqref{eqs:sys3} by means of polar coordinates transformation (see Appendix \ref{app:Polar} for details). First, expand the current dynamics in \eqref{eqs:sys3}:
\begin{subequations}\label{eqs:i_dq dot}
\begin{align}
\ell\dot{i}_\tx{d}&=\mu_\tx{r} v_\tx{dc}\cos(\theta)-ri_\tx{d}+\ell\omega_0i_\tx{q}-v_\tx{d},
\\
\ell\dot{i}_\tx{q}&=\mu_\tx{r} v_\tx{dc}\sin(\theta)-ri_\tx{q}-\ell\omega_0i_\tx{d}-v_\tx{q}.
\end{align}
\end{subequations}
Consider that $\n{i}=\sqrt{i_\tx{d}^2+i_\tx{q}^2}$ and thus
\begin{equation}\label{eq:scalar ODE00}
\dt{\n{i}^2}=2\n{i}\dt{\n{i}}=2\big(i_\tx{d}\dot{i}_\tx{d}+i_\tx{q}\dot{i}_\tx{q}\big).
\end{equation}
Replace $i_\tx{d}$ and $i_\tx{q}$ with $\n{i}\cos(\theta_i)$ and $\n{i}\sin(\theta_i)$ where $\theta_i \coloneqq \tan^{-1}(i_\tx{q}/i_\tx{d})$ and multiply \eqref{eq:scalar ODE00} with $\ell/2\n{i}$:
\begin{equation}\label{eq:scalar ODE0}
\ell\dt{\n{i}}=\ell\cos(\theta_i)\dot{i}_\tx{d}+\ell\sin(\theta_i)\dot{i}_\tx{q}.
\end{equation}
Next, substitute $v_\tx{d}$ and $v_\tx{q}$ in \eqref{eqs:i_dq dot} with $\n{v}\cos(\theta_v)$ and $\n{v}\sin(\theta_v)$ where $\theta_v \coloneqq \tan^{-1}(v_\tx{q}/v_\tx{d})$. 
Replacing $\ell\dot{i}_\tx{d}$ and $\ell\dot{i}_\tx{q}$ in \eqref{eq:scalar ODE0} with the \ac{rhs} of \eqref{eqs:i_dq dot} and exploiting \eqref{id: cosine angle sum} results in
\begin{equation}\label{eq:scalar ODE}
\ell\dt{\n{i}}=\mu_\tx{r} v_\tx{dc}\cos(\theta-\theta_i)-r\n{i}-\n{v}\cos(\theta_v-\theta_i)\,.
\end{equation}
So far $\mu_\tx{r}$ in \eqref{eq:scalar ODE} was assumed to be a positive constant; see the definition of $m_{\alpha\beta}(\mu_\tx{r},\theta_\tx{c})$ in \eqref{eqs:converter}. We  now re-parametrize the to-be-controlled modulation magnitude as
\begin{equation}\label{eq:mu}
\mu \coloneqq \big(1-\Delta_\mu\big)\mu_\tx{r},
\end{equation}
where $\Delta_\mu:\mathbb{X}\rightarrow\mathbb{R}_{(0,1)}$ is a locally Lipschitz function specified later. Replacing $\mu_\tx{r}$ in \eqref{eq:scalar ODE} with \eqref{eq:mu} results in
\begin{equation}\label{eq:scalar ODE2}
\ell \dt{\n{i}}=\mu_\tx{r} v_\tx{dc}\cos(\theta-\theta_i) \big(1-\mathcal{D}-\Delta_\mu\big)-r\n{i},
\end{equation}
where 
\begin{equation}\label{eq:D}
\mathcal{D} \coloneqq \dfrac{\n{v}\cos(\theta_v-\theta_i)}{\mu_\tx{r}v_\tx{dc}\cos(\theta-\theta_i)}
\end{equation}
takes value in $\mathbb{R}_{(0,1)}$ under normal operation and can be constructed from current and voltage measurements (the cosines of the angle differences in \eqref{eq:D} can be computed with similar techniques as in the Section \ref{sec:implementation}).  

In the sequel, we treat $\mc{D}$ as a fictitious disturbance -- capturing the influence of other states -- in magnitude dynamics \eqref{eq:scalar ODE2}. Consider a  threshold current $i_\tx{th}\in\mathbb{R}_{>0}$ that $\n{i}$ should not exceed. We aim to design a $\Delta_\mu$ such that 1) the \ac{rhs} of \eqref{eq:scalar ODE2} is strictly negative for all $\n{i}>i_\tx{th}$ and 2) ideally (i.e., not necessarily) $\Delta_\mu= 0$ for $\n{i}\leq i_\tx{th}$. The design of $\Delta_\mu$ in Proposition \ref{prop:current limitation} is inspired by ideas from feedback linearization \cite[Chap. 13]{khalil_nonlinear_2002} and disturbance decoupling techniques. Moreover, our design follows the intuition that for limiting the ac current, the dc power (i.e., $p_\tx{net}$ in Figure \ref{fig:SCIBsys}) must be constrained by controlling the modulation magnitude. 
\begin{proposition}\textup{{(Bivariate current-limiting control)}}\label{prop:current limitation}\\
Consider the current magnitude dynamics \eqref{eq:scalar ODE2} and assume that $|\theta-\theta_i|<\pi/2$ and $\mc{D}\in\mathbb{R}_{(0,1)}$. Define
\begin{equation}\label{eq:delta_mu}
\Delta_\mu \coloneqq \dfrac{\left(1-\mc{D}\right)e^{\beta(\n{i}-i_\tx{th})}}{1+\left(1-\mc{D}\right)\left(e^{\beta(\n{i}-i_\tx{th})}-1\right)}\,,
\end{equation}
with $\beta\in\mathbb{R}_{>0}$, then $\n{i}$ is strictly decreasing for $\n{i}>i_\tx{th}$. 
\end{proposition}
\begin{proof}
Define $\mc{C} \coloneqq 1-\mc{D}$ and observe that 
\begin{equation*}
\underset{\substack{(\n{i},\mc{C})\to\left(i_\tx{th}^+,0^+\right)}}{\lim}\Delta_\mu=0\quad\tx{and}\quad \underset{\substack{(\n{i},\mc{C})\to\left(+\infty,1^-\right)}}{\lim}\Delta_\mu=1\,.
\end{equation*}
Moreover, $\Delta_\mu$ is strictly increasing in both $\n{i}$ and $\mc{C}$ i.e.,
\begin{subequations}\label{eqs:partials of Delta_mu}
\begin{align}
\dfrac{\partial\Delta_\mu}{\partial\n{i}}&=\dfrac{\beta\mc{C}(1-\mc{C})e^{\beta(\n{i}-i_\tx{th})}}{\left(1+\mc{C}\left(e^{\beta(\n{i}-i_\tx{th})}-1\right)\right)^2}>0\,,
\\
\dfrac{\partial\Delta_\mu}{\partial\mc{C}}&=\dfrac{e^{\beta(\n{i}-i_\tx{th})}}{\left(1+\mc{C}\left(e^{\beta(\n{i}-i_\tx{th})}-1\right)\right)^2}>0\,.
\end{align}
\end{subequations}
Since $\Delta_\mu$ is strictly monotone and continuous with finite limits, then it is bounded by its left and right limits i.e., $\Delta_\mu\in(0,1)$ for all $\mc{C}\in(0,1)$ and $\n{i}>i_\tx{th}$. Further, since $\Delta_{\mu}|_{\n{i}=i_\tx{th}}=\mc{C}$ and \eqref{eqs:partials of Delta_mu} hold then $\mc{C}<\Delta_\mu<1$ for any $\mc{C}$ and $\n{i}>i_\tx{th}$. Thus, by the assumptions on $\theta-\theta_i$, and with $\Delta_\mu$ as in \eqref{eq:delta_mu}, the \ac{rhs} of \eqref{eq:scalar ODE2} is strictly negative for $\n{i}>i_\tx{th}$. Nagumo's theorem \cite[Th. 3.1]{B99} yields that $\n{i(t)}$ is strictly decreasing whenever $\n{i(t)}>i_\tx{th}$.  
\end{proof}
Observe that $\Delta_\mu\in\mathbb{R}_{(0,1)}$ is required since $\mu$ in \eqref{eq:mu} has to be positive in practice. Subsequently, the assumption that $\mc{D}\in\mathbb{R}_{(0,1)}$ is vital; otherwise, the current magnitude in \eqref{eq:scalar ODE2} cannot be bounded by controlling $\mu$ via bounded $\Delta_\mu$. 
\begin{figure}[b!]
	\centering
	\includegraphics[width=1.02\columnwidth]{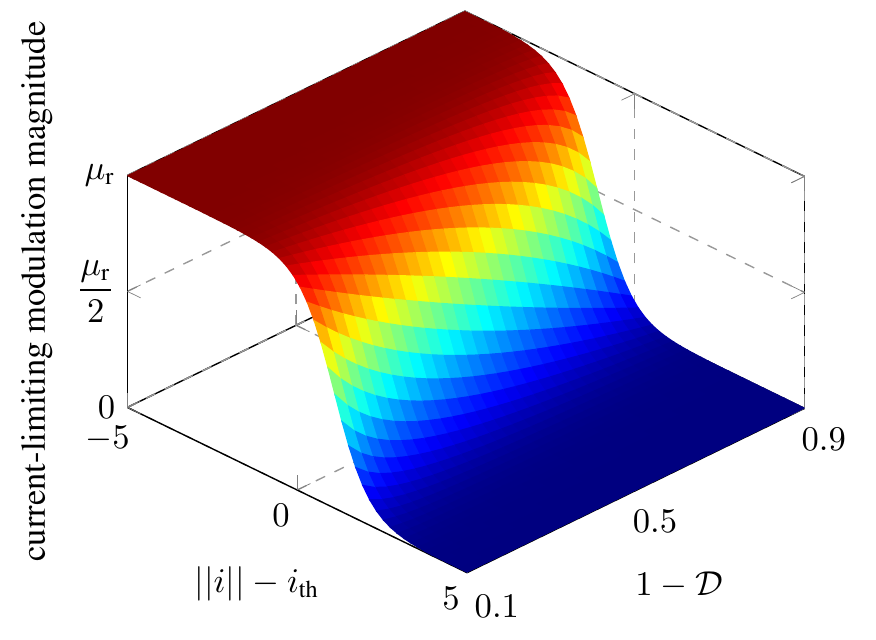}
	\caption{Illustration of $\mu$ in  \eqref{eq:mu} with $\Delta_\mu$ in \eqref{eq:delta_mu}. For clarity of presentation, here $\beta=2$ and $\mc{D}\in\mathbb{R}_{(0.1,0.9)}$. For small $\mc{D}$ (i.e., a severe contingency) $\Delta_\mu$ initiates the modulation magnitude decay at a lower current compared to a scenario with large $\mc{D}$. \label{fig:Delta}}
\end{figure}
\begin{remark}\textup{(Comments on the current-limiting control)}\\
The following comments are in order:

First, the bivariate function $\Delta_\mu$ in \eqref{eq:delta_mu} should be understood as  a barrier-type function that reduces $\mu$ in \eqref{eq:mu} and thus the switching voltage magnitude $\n{v_\tx{s}}$ (see Figure \ref{fig:SCIBsys}) when $\n{i}>i_\tx{th}$; observe the influence of $\Delta_\mu$ on $\mu$ in Figure \ref{fig:Delta}. The maximum reduction of $\mu$ occurs as $\mc{D}\to 0$  corresponding to a severe contingency e.g., a fault at the filter capacitance node driving $\n{v}\to 0$ (see Subsection \ref{subsec:current-limiting example} for an example). 

Second, the parameter $\beta$ controls the curvature and the exponential decay rate of $\mu$ in $\n{i}$. Note that ${\Delta_\mu}\to0$ as $\beta\to+\infty$ for $\n{i}<i_\tx{th}$ and for any $\mc{D}$, however large $\beta$ results in an aggressive controller resembling a sign function. 

Third, practical implementation of $\mathcal{D}$ in \eqref{eq:D} can be prone to measurement imperfections. Therefore, if $\mc{D}$ admits a non-zero lower bound i.e., $0<\mc{D}_{{\tx{min}}}<\mc{D}$, it allows to implement a variant of \eqref{eq:delta_mu} that is independent of $\mc{D}$. Such disturbance-free implementation of $\Delta_\mu$ is obtained by replacing $\mathcal{D}$ with $\mc{D}_{{\tx{min}}}$ in \eqref{eq:delta_mu}. In practice, $\mc{D}_\tx{min}$ is chosen by estimating $\mc{D}$ for the worst-case scenario. 

Last, the result of the Proposition \ref{prop:current limitation} can be improved when augmenting $\mc{D}$ with the dissipation term $r\n{i}$ in \eqref{eq:scalar ODE2}. However, in practice $r\n{i}$ is negligible compared to the denominator of \eqref{eq:D} (because of insignificant resistance e.g., see Table \ref{tab:parameters}).   
\end{remark}
The assumptions in Proposition \ref{prop:current limitation} (i.e., bounded angle and disturbance) are standard in small-signal / input-to-state stability and protection design. Moreover, the bound on the disturbance feasible set i.e., $\mc{D}\in\mathbb{R}_{(0,1)}$ can actually be extended to $\mc{D}\in\mathbb{R}_{(0,2)}$. To make this idea clear for now assume that $\mu=\mu_\tx{r}$, then multiplying the nominator and denominator \eqref{eq:D} with $\n{i}$ gives $\mc{D}=p_\tx{f}/p_\tx{s}$ (see Figure \ref{fig:SCIBsys}). Thus, $\mc{D}>1$ equals $p_\tx{f}>p_\tx{s}$ corresponding to the -- less likely -- scenario in which converter absorbs power from grid, e.g., after loss of load. In this case, replacing $1-\mc{D}$ with $|1-\mc{D}|$ in \eqref{eq:delta_mu} (while preserving the local Lipschitz continuity of $\Delta_\mu$) guarantees that the \ac{rhs} of \eqref{eq:scalar ODE2} is negative for $\mc{D}\in\mathbb{R}_{(0,2)}$. 

It remains to be shown that tampering with the current magnitude in \eqref{eq:scalar ODE2} does not jeopardize the overall system stability.
Proposition \ref{prop:hac and current limitation} gives an affirmative answer: under \eqref{eq:stability condition} and with current-limiting control \eqref{eq:mu} the desired closed-loop equilibrium of \eqref{eqs:sys} remains locally asymptotically stable. 
\begin{proposition}\textup{{(\ac{hac} and current-limiting control)}}\label{prop:hac and current limitation}\\
Consider the closed-loop system \eqref{eqs:sys} where $\mu_\tx{r}$ is replaced by the bounded $\mu$ in \eqref{eq:mu}. Assume that the modified dynamics admits equilibria  of the form \eqref{eq:EqSet} and condition \eqref{eq:stability condition} holds. Then the equilibrium $x^\star_\tx{s}$ is locally asymptotically stable.
\end{proposition}
\begin{proof}
By replicating the proof of Theorem \ref{TH1} it follows that
\begin{align}
\dot{\mathcal{V}}(\tilde{x})\leq&-\left(\lambda\gamma-\dfrac{1}{\epsilon_1^2}-\dfrac{(\mu v_\tx{dc}^\star)^2}{\epsilon_2^2}\right) \sin^2\big(\tilde{\theta}/2\big)-\dfrac{1}{\kappa}\tilde{i}^2_\tx{dc}
\nonumber
\\
&-\left(g_\tx{dc}-\left({\epsilon_1\mu \n{i^\star}}\right)^2\right)\tilde{v}^2_\tx{dc}-\left(r-\epsilon_2^2\right)\n{\tilde{i}}^2 
\nonumber
\\
&- g\n{\tilde{v}}^2-r_\tx{g}\n{\tilde{i}_\tx{g}}^2+\lambda\eta\tilde{v}_\tx{dc}\sin\big(\tilde{\theta}/2\big)\,.
\label{eq:Vdot bound with mu}
\end{align}
Since $\mu<\mu_\tx{r}$, the \ac{rhs} of \eqref{eq:Vdot bound with mu} is smaller than the bound in \eqref{eq:Vdot bound}, therefore $\dot{\mc{V}}(\tilde{x})\leq-\tilde{\zeta}^\top\m{Q}\tilde{\zeta}$ as in \eqref{eq:Vdot bound}. Note that $\mc{V}(0_9)=0$, $\mc{V}(\tilde{x})>0$ for all $\tilde{x}\neq0_9$. By the proof of Theorem \ref{TH2} if \eqref{eq:stability condition} holds, $\dot{\mc{V}}(\tilde{{x}})<0$ for all $\tilde{{x}}\neq0_9$ in a sufficiently small neighborhood of the origin. The local asymptotic stability of $x^\star_\tx{s}$ immediately follows from Lyapunov's direct method \cite[Th. 3.1]{khalil_nonlinear_2002} as in the proof of Proposition \ref{prop:LAS}.
\end{proof}
It is worth mentioning that the current-limiting control \eqref{eq:mu} is agnostic to the \ac{hac}. Hence, it is expected that \eqref{eq:mu} is practically compatible with different grid-forming controls in \cite{TGAKD20} (although possibly without any stability guarantees). The performance of control \eqref{eq:mu} for a three-phase-to-ground fault scenario is verified in the Subsection \ref{subsec:current-limiting example}.
\section{HAC Implementation and Droop Behavior}	\label{sec:implementation}
In what follows, we describe a practical \ac{hac} implementation, introduce a feedforward ac voltage and power control, and unmask the power-frequency droop behavior of \ac{hac}.
\subsection{Implementation of \ac{hac} and Feedforward Control}
The formulation of \ac{hac} in \eqref{eq:omega_c} relies on the explicit relative angle reference and measurement that are not available in practice. Hence, we seek an alternative implementation based on the dc and ac voltage measurements, and given set-points. 

The dc term in \eqref{eq:omega_c} is constructed by measuring the dc voltage; see Figure \ref{fig:SCIBsys}. The \ac{ib} voltage is also measured and transformed to $\alpha\beta$-frame (see Appendix \ref{app:Polar} for details). Subsequently, an implicit \ac{ib} angle measurement is obtained i.e., $\psi(\theta_\tx{b})=v_{\tx{b,}\alpha\beta}/\n{v_{{\tx{b,}\alpha\beta}}}$. Note that $\psi(\theta_\tx{c})$ is internally available from the modulation vector $m$. Thus, by the means of \eqref{id: sine angle sum} and \eqref{id: cosine angle sum},  an implicit relative angle measurement $\psi(\theta)$ is derived 
\begin{equation}\label{eq:psitheta}
\psi(\theta)=
\begin{pmatrix}
\psi(\theta_\tx{c})^\top \psi(\theta_\tx{b})
,
\psi(\theta_\tx{c})^\top\m{J} \psi(\theta_\tx{b}) 
\end{pmatrix}.
\end{equation}  
Given a relative angle reference $\psi(\theta_\tx{r})$ and measurement \eqref{eq:psitheta}, Proposition \ref{prop:implementation} explains the derivation of angle feedback in \eqref{eq:omega_c}.
\begin{figure*}[t!]
\centering	
{\includegraphics[trim=7.2mm 0mm 0mm 0mm,clip,width=0.9\textwidth]{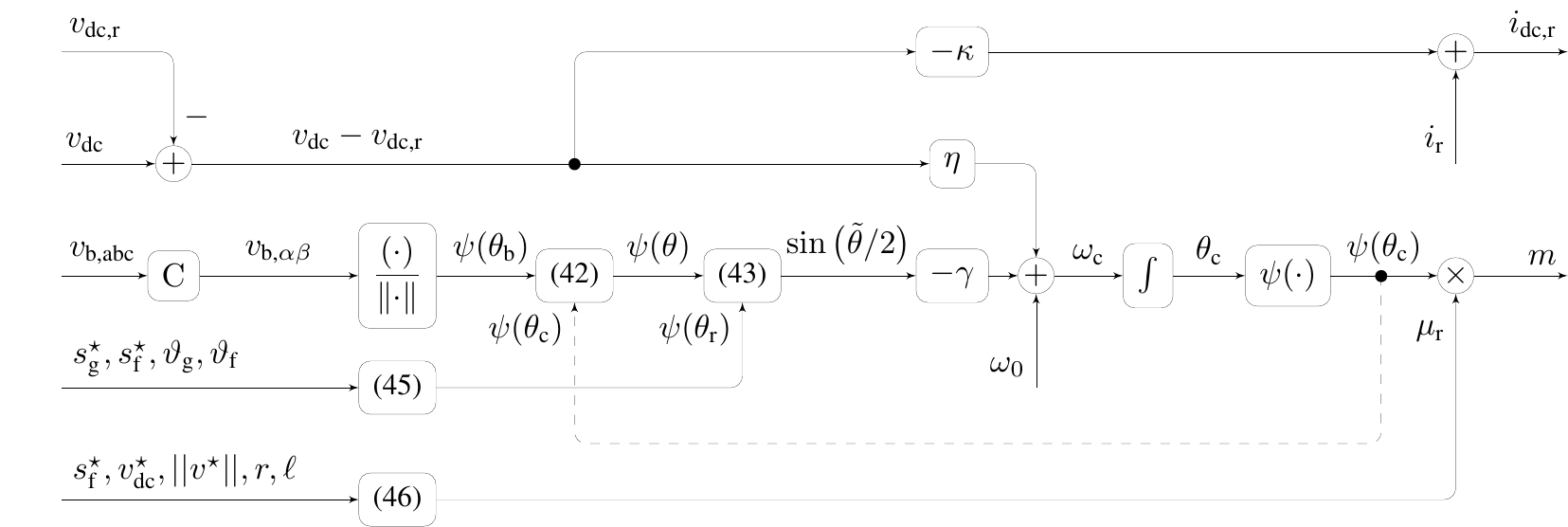}}
\caption{Block diagram of the feedback controls \eqref{eq:i_dc} and \eqref{eq:omega_c} with implementation \eqref{eq:impl}, in combination with the feedforward controls \eqref{eq:psi(theta_r)} and \eqref{eq:mu_r} for the closed-system \eqref{eqs:sys}.\label{fig:block diagram}}
\end{figure*}
\begin{proposition}\textup{{(Angle feedback implementation)}}	\label{prop:implementation}		\\
	Given $\psi(\theta_\tx{r})$ and $\psi(\theta)$ derived by \eqref{eq:psitheta}, if $|\theta-\theta_\tx{r}|<\pi$ then
	\begin{equation}\label{eq:impl}
	\sin\left( \dfrac{ \theta - \theta_\tx{r} } { 2 } \right) = \dfrac {\psi(\theta_\tx{r})^{\top}\m{J} \psi(\theta)} 
	{\sqrt { 2 \big( 1 + \psi(\theta_\tx{r})^{\top} \psi(\theta) \big) }}\,.
	\end{equation}
\end{proposition}	
\begin{proof}
	Consider $\psi(\theta_\tx{r})^{\top} \m{J} \psi(\theta)$ and apply \eqref{id: sine angle sum}:
	\begin{align*}
	\psi(\theta_\tx{r})^{\top} \m{J} \psi(\theta) &=
	\sin(\theta) \cos(\theta_\tx{r}) - \sin(\theta_\tx{r}) \cos(\theta)
	\\
	&=\sin\big(\tilde{\theta}\big)= 2 \sin\big(\tilde{\theta}/2\big) \cos({\tilde{\theta}}/{2}),
	\end{align*}
	where $\tilde{\theta} \coloneqq \theta-\theta_\tx{r}$. Next, consider $\psi(\theta_\tx{r})^{\top} \psi(\theta)$ and apply \eqref{id: cosine angle sum}:
	\begin{equation*}
	\psi(\theta_\tx{r})^{\top} \psi(\theta) = \cos(\theta) \cos(\theta_\tx{r}) + \sin(\theta) \sin(\theta_\tx{r}) = \cos\big(\tilde{\theta}\big).
	\end{equation*}
	Subsequently, applying identity \eqref{id: cosine half angle} results in 
	\begin{equation*}
	\sqrt{ 2 \big( 1 + \psi(\theta_\tx{r})^{\top} \psi(\theta) \big) } = \sqrt{ 2 \big( 1 + \cos\big(\tilde{\theta}\big) \big) } = 2 \big| \cos\big(\tilde{\theta}/2\big) \big|.
	\end{equation*}
Hence, the \ac{rhs} of \eqref{eq:impl} is equal to	%
\begin{equation}
\dfrac { \cos \big( \tilde{\theta} / 2 \big) \sin \big( \tilde{\theta} / 2 \big) } { \big| \cos \big( \tilde{\theta} / 2 \big) \big| } = \mathrm{ sgn } \big( \cos \big( \tilde { \theta } / 2 \big) \big) \sin \big( \tilde { \theta } / 2 \big).
\label{eq:impl-rhs}
\end{equation}
Thus, if $| \tilde{\theta} | < \pi$ then $\mathrm{sgn} \big( \cos\big({\tilde{\theta}}/{2}\big) \big) = 1$ and \eqref{eq:impl} holds. 
\end{proof}
As well as $\psi{ ( \theta_\tx{r} ) }$ in Proposition \ref{prop:implementation}, the reference modulation magnitude $\mu_\tx{r}$ (see Figure \ref{fig:SCIBsys}) is not explicitly available in practice. Rather power references $p_\tx{g,r}$ and $q_\tx{g,r}$ (associated with $p_\tx{g}$ and $q_\tx{g}$ in Figure \ref{fig:SCIBsys}) are specified. In what follows, we describe consistent definitions for $\psi{ ( \theta_\tx{r} ) }$ and $\mu_\tx{r}$ that rely on the steady-state dc voltage, ac voltage magnitude, and power flows (see \eqref{eqs:p*_g,q*_g} and \cite[Def. 2]{CGBF19} for a definition). 
\begin{proposition}\textup{{(Consistent definition of $\psi(\theta_\tx{r})$ and $\mu_\text{r}$)}}	\label{prop:consistent references}\\
	Given the voltages $v^\star_\tx{dc}$, $\n{v^\star}$, and the references $p_\tx{g,r}$ and $q_\tx{g,r}$ consistent with the equilibrium values i.e., $p_\tx{g,r}=p_\tx{g}^\star$ and $q_\tx{g,r}=q_\tx{g}^\star$, the consistent references are defined by  
	\begin{align}
	\psi(\theta_\tx{r}) &    \coloneqq    \m{R}(\delta)^\top 
	\Big( 
	\hat{s}_\tx{g}^{\star\top} \left(\begin{smallmatrix} +1 & 0 \\ 0 & -1 \end{smallmatrix}\right) \hat{s}^\star_\tx{f}
	, 
	\hat{s}_\tx{g}^{\star\top} \left(\begin{smallmatrix} 0 & +1 \\+1 & 0 \end{smallmatrix}\right) \hat{s}^\star_\tx{f} 
	\Big),
	\label{eq:psi(theta_r)}
	\\	
	\mu_\tx{r}  &    \coloneqq   { \sqrt{\big(p^{\star 2}_\tx{f}+q^{\star 2}_\tx{f}\big) \big(r^{2} + ( \ell \omega_0 )^{2}\big) }}\Big/{{ v^\star_\tx{dc}\n{v^\star} }},
	\label{eq:mu_r}
	\end{align}
	with $\delta \coloneqq \tan^{-1}(\ell_\tx{g}\omega_0/r_\tx{g})+\tan^{-1}(\ell\omega_0/r)$, $s^\star_\tx{g}   \coloneqq   (p^\star_\tx{g},q^\star_\tx{g})$ and $\hat{s}^\star_\tx{g}   \coloneqq   s^\star_\tx{g}/\n{s^\star_\tx{g}}$, $s^\star_\tx{f}   \coloneqq   (p^\star_\tx{f},q^\star_\tx{f})$ and $\hat{s}^\star_\tx{f}   \coloneqq   s^\star_\tx{f}/\n{s^\star_\tx{f}}$. 
\end{proposition}
The proof is given in Appendix \ref{App:proofs}. Observe that the consistent reference specifications \eqref{eq:psi(theta_r)}-\eqref{eq:mu_r} can conceptually also be used as feedforward ac voltage and power control.
Figure \ref{fig:block diagram} represents the overall block diagram of the feedback controls \eqref{eq:i_dc}, \eqref{eq:omega_c}, and the feedforward controls \eqref{eq:psi(theta_r)} and \eqref{eq:mu_r}. 

If the assumption $|\theta-\theta_\tx{r}|<\pi$ in Proposition \ref{prop:implementation} is not met,  then according to \eqref{eq:impl-rhs} the \ac{rhs} of \eqref{eq:impl} equals a $2\pi$-periodic switching signal 
\begin{equation}\label{eq:u_sw}
u_\tx{sw} \coloneqq \mathrm{sgn} \big( \cos\big(\tilde{\theta}/2\big) \big) \sin\big(\tilde{\theta}/2\big),
\end{equation}
where $\tilde{\theta} \coloneqq \theta-\theta_\tx{r}$. Remark \ref{rem:switching} explains the implications of \eqref{eq:u_sw} for the closed-loop dynamics \eqref{eqs:sys}. 
\begin{figure}[b!]
\centering
{\includegraphics[trim=2mm 2mm 1mm 1mm,clip, width=0.85\columnwidth]{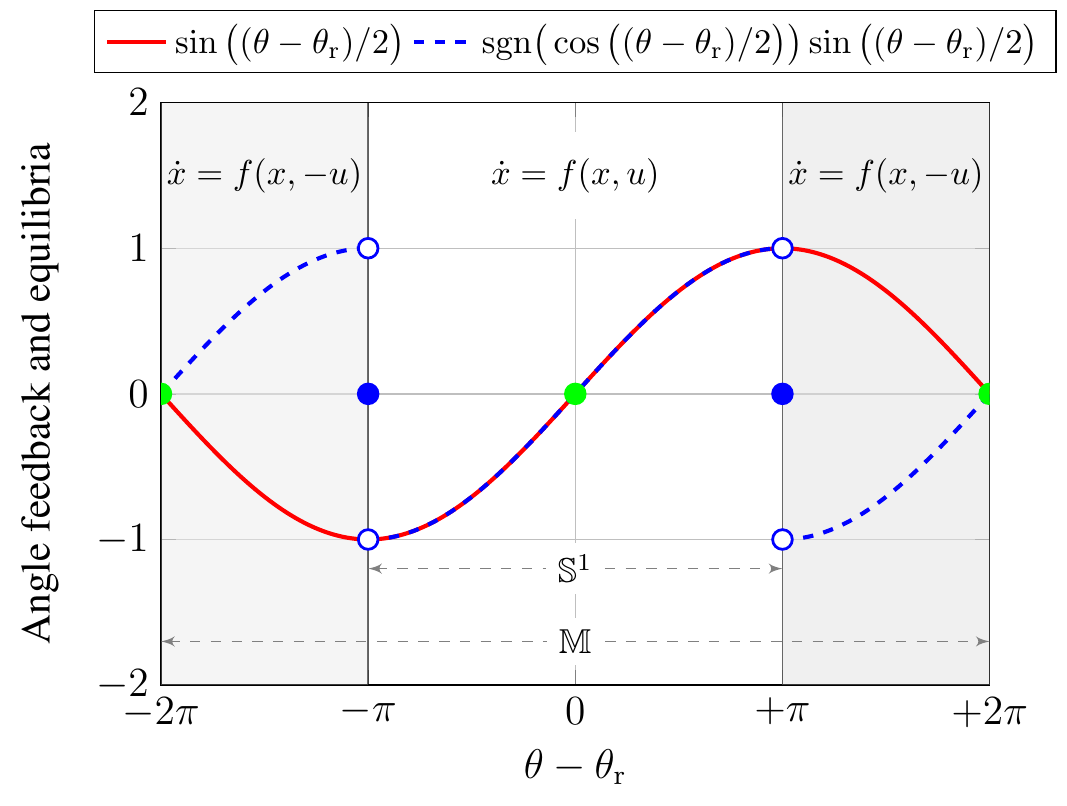}}
\caption{Illustration of the angle feedback in \eqref{eq:omega_c} and the implementation \eqref{eq:u_sw}. The implementation introduces two new equilibria (in blue) aside the existing	ones (in green) in $\mathbb{M}$.\label{fig:implementation}}
\end{figure}
\begin{remark}\textup{{(Switching feedback and angle invariance)}}\label{rem:switching}\\
The following comments are in order:

First, the feedback $u_\tx{sw}$ in \eqref{eq:u_sw} is $2\pi$-periodic since the implicit angle information in \eqref{eq:impl} (that is embedded in $\psi(\theta_\tx{r})$ and $\psi(\theta)$) is confined to $\mathbb{S}^1$ due to $2\pi$-periodicity of $\psi(\cdot)$. 

Second, an exact implementation of $\sin\big(\tilde{\theta}/2\big)$ based on \eqref{id: sine half angle} with $\cos\big(\tilde{\theta}\big)=\psi(\theta_\tx{r})^{\top} \psi(\theta)$ (as in proof of Proposition \ref{prop:implementation}) can be obtained
\begin{equation*}
\sin\big({\tilde{\theta}}/{2}\big)   \coloneqq   (-1)^{\lfloor{\tilde{\theta}}/{(2\pi)}\rfloor}\big|\sin\big(\tilde{\theta}/2\big)\big|,
\end{equation*}
that requires $4\pi$-periodic angle error information to construct the appropriate sign pattern. However, in practice the explicit measurement of $\tilde{ \theta }$ is not possible. 

Third, the sign function in \eqref{eq:u_sw} introduce new equilibria for the closed-loop system \eqref{eqs:sys} at $\tilde{\theta}=\pm \pi$ (see Figure \ref{fig:implementation}). A separate analysis reveals the unstable nature of these equilibria that divide $\mathbb{M}$ into two regions with switched subsystems i.e.,
	\begin{equation}\label{eqs:SwitchingSys}
	\dot{x}=\begin{cases}
	f(x,u)&x\in~]-\pi,\pi[~\times~\mathbb{R}^8,\\
	f(x,-u)&x\in~]-2\pi,-\pi[~\cup~]\pi,2\pi[~\times~\mathbb{R}^8,\\
	\end{cases}
	\end{equation}
where $f(x,u)$ denotes the vector field \eqref{eqs:sys} with $u \coloneqq \sin\big(\tilde{\theta}/2\big)$. An analysis  similar to Section \ref{sec:closed-loop analysis}, verifies the \ac{agas}  of the second subsystem i.e., $\dot{x}=f(x,-u)$, $x\in\mathbb{X}$ \ac{wrt} $x^\star_\tx{u}$ in \eqref{eq:EqSet}. Thus, if $x(t)$ hits the switching surface $\Omega_\tx{sw} \coloneqq \{x\in\mathbb{X}:|\tilde{\theta}|=\pi\}$ it triggers the switch between the subsystems in \eqref{eqs:SwitchingSys} that are individually \ac{agas}  \ac{wrt} either $x^\star_\tx{s}$ or $x^\star_\tx{u}$. The analysis of switched system \eqref{eqs:SwitchingSys} requires a separate study outside the scope of this paper. 

Finally, we note that	 if $\eta=0$ in \eqref{eq:omega_c}, then $\dot{\theta}=-\gamma u_\tx{sw}$ and  $\mathscr{L}_{\tilde{\theta}} \coloneqq \{\tilde{\theta}\in\mathbb{S}^1:|\tilde{\theta}|<\pi\}$ is invariant under the flow defined by $\dot{x}=f(x,u_\tx{sw})$. 	Thus, if $\tilde{ \theta }(0)\in\mathscr{L}_{\tilde{\theta}}$, $x(t)$ never reach $\Omega_\tx{sw}$. 
\end{remark}
Remark \ref{rem:atan} presents an alternative angle feedback that, in theory, can be incorporated in \ac{hac} \eqref{eq:omega_c}.
\begin{remark}\textup{{(A globally stabilizing feedback control)}}\label{rem:atan}\\
In the spirit of hybrid angle control, an alternative to \eqref{eq:omega_c} is  
\begin{equation}\label{eq:u tan}
\omega_\tx{c} \coloneqq \omega_{0} + \eta ( v_\tx{dc}  - v_\tx{dc,r} ) - \gamma \tan^{-1}( \theta - \theta_\tx{r} ).
\end{equation} 
Under \eqref{eq:u tan} the closed-loop system \eqref{eqs:sys} is not periodic in angle and its solutions evolve in $\mathbb{R}^9$. If $v^\star_\tx{dc}=v_\tx{dc,r}$ (see Theorem \ref{TH1}) the unique angle equilibrium is identified by $\gamma\tan^{-1}\left({\theta^\star-\theta_\tx{r}}\right)=0$. Moreover, if $\mathcal{S}\big(\tilde{\theta}\big)$ in Lyapunov function \eqref{eq:LF} is replaced by
\begin{equation}
\mathcal{\ubar{S}}\big( \tilde{\theta} \big) = \int_{0}^{\tilde{\theta}} \tan^{-1}( s ) \tx{d}s = \tilde{\theta} \tan^{-1}\big( \tilde{\theta} \big) - \dfrac{ 1 }{ 2 } \ln \big( 1 + \tilde{\theta}^2 \big),
\end{equation}
a similar analysis as in the Theorem \ref{TH2} uncovers the \underline{global asymptotic stability} (rather than \ac{agas}) of the unique equilibrium on the basis of Lyapunov's direct method. Nonetheless, it is not clear how to construct $\tan^{-1} \big( \tilde{\theta} \big)$ from $\psi(\theta)$ and $\psi(\theta_\tx{r})$ in Propositions \ref{prop:implementation} and \ref{prop:consistent references}. 
\end{remark}
\subsection{\ac{hac} Power-Frequency Droop Behavior}\label{subsec:droop}
In what follows, we consider the converter-\ac{coi} closed-loop dynamics \eqref{eqs:sys COI} that allows frequency droop which is hindered when considering the \ac{ib} grid model. Recall that the existence of equilibria $\ubar{\Omega}^\star$ in \eqref{eq:EqSetCOI} is guaranteed if $v_\tx{dc}^\star=v_\tx{dc,r}$ and $\omega^\star=\omega_0$. These assumptions can be met by appropriate choice of $i_\tx{r}$ and $T_\tx{m}$ in \eqref{eqs:sys COI2} and \eqref{eqs:sys COI6}. For the sake of argument, assume that $i_\tx{r}$ and $T_\tx{m}$ are not consistent with the assumptions or the system is subject to disturbances, then the dc voltage and frequency converge to different equilibria $v_\tx{dc,x}$ and $\omega_\tx{x}$. Hence, by \eqref{eqs:sys COI1} the relative angle settles at a different equilibrium $\theta_\tx{x}\neq\theta_\tx{r}$. The drift from desired references is also reflected in the ac quantities. 

Inspired by \cite[Prop. 5]{AJD18}, in Proposition \ref{prop:droop slope}  below we derive the droop slope that relates the active power and frequency at an arbitrary equilibrium. More precisely, for an operating frequency $\omega_\tx{x}\in\mathbb{R}_{>0}$, the power-frequency linear sensitivity factor (also termed {\em droop}) is defined by
$d_{p-\omega} \coloneqq {\partial p_\tx{net,x}}/{\partial\omega_\tx{x}}$;
see Figure \ref{fig:SCIBsys} and the Subsection \ref{subsec:SCIB} for the definition of $p_\tx{net,x}$. 
\begin{proposition}\textup{{(Power--frequency droop slope)}}\label{prop:droop slope}\\
Consider system \eqref{eqs:sys COI}, the power-frequency droop slope at an equilibrium with frequency $\omega_\tx{x}$ and dc voltage $v_\tx{dc,x}$ equals
\begin{equation}\label{eq:droop slope}
d_{ p-\omega } = - \left( \dfrac{ 2 G_\tx{dc} } { \eta^2 } \right) \omega_\tx{x} + \left( \dfrac{ \eta i_0 + 2 G_\tx{dc}\beta_{ \theta_\tx{x} } } { \eta^2 } \right),
\end{equation} 
where $G_\tx{dc}   \coloneqq   \kappa + g_\tx{dc}$, $ i_0 \coloneqq i_\tx{r} + \kappa v_\tx{dc,r} $, $ \beta_{\theta_\tx{x}}   \coloneqq   \omega_0  - \eta v_\tx{dc,r}  - \gamma \sin\big( (\theta_\tx{x}  - \theta_\tx{r} ) / 2 \big) $, and $\theta_\tx{x}$ is equilibrium relative angle. 
\end{proposition}
The proof is given in Appendix \ref{App:proofs}. 
%
\begin{figure*}[b!]
	\hspace{0.14mm}
	{\includegraphics[trim=3.5cm 1cm 3cm 1cm ,clip,width=1.017\textwidth]{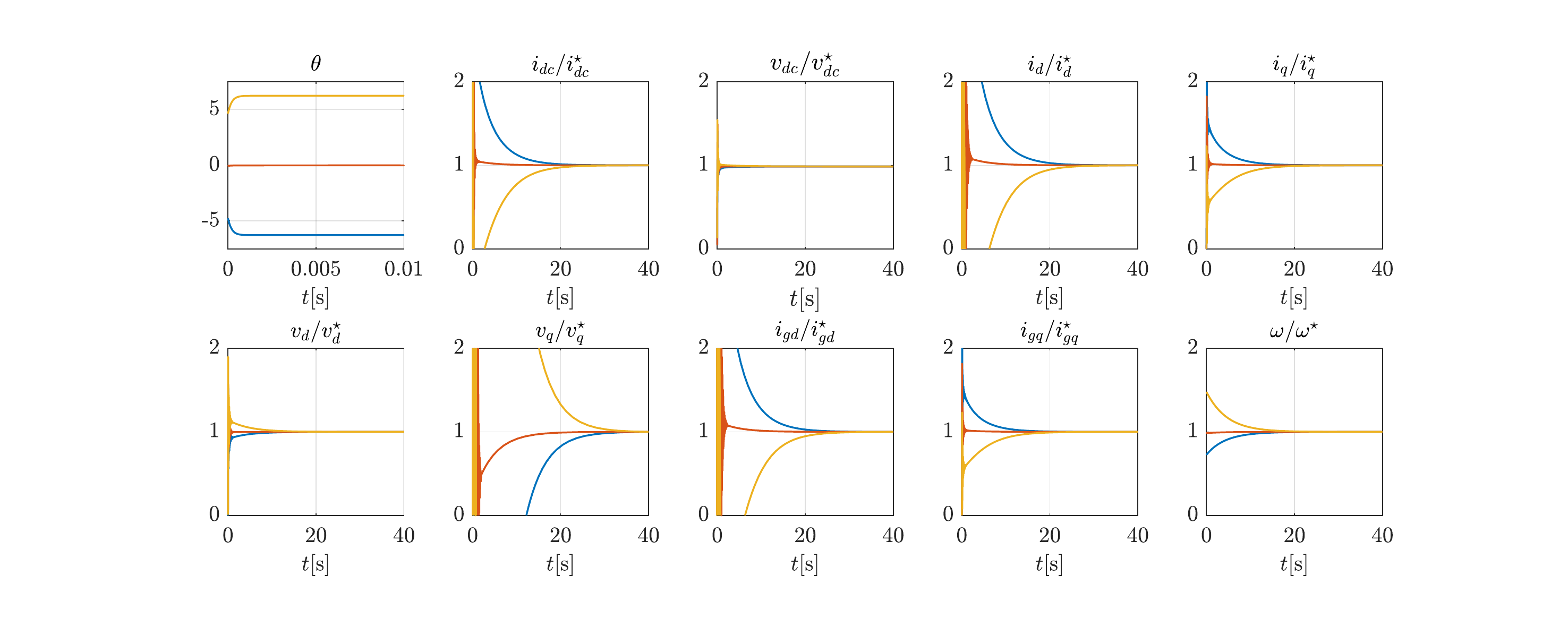}}
	\caption{Per-unit-normalized trajectories of the closed-loop system \eqref{eqs:sys COI} with the \ac{hac} implementation in Proposition \ref{prop:implementation}, parameters in Table \ref{tab:parameters} and three different initial conditions.}\label{fig:timeseries}
\centering
{\includegraphics[trim=6mm 6mm 6mm 6mm ,clip,width=0.18\textwidth]{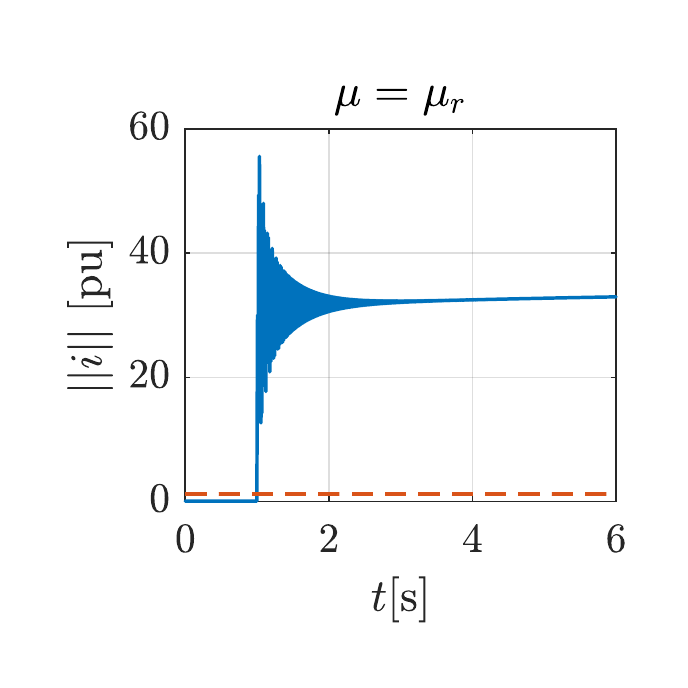}}\hfill
{\includegraphics[trim=6mm 6mm 6mm  6mm ,clip,width=0.18\textwidth]{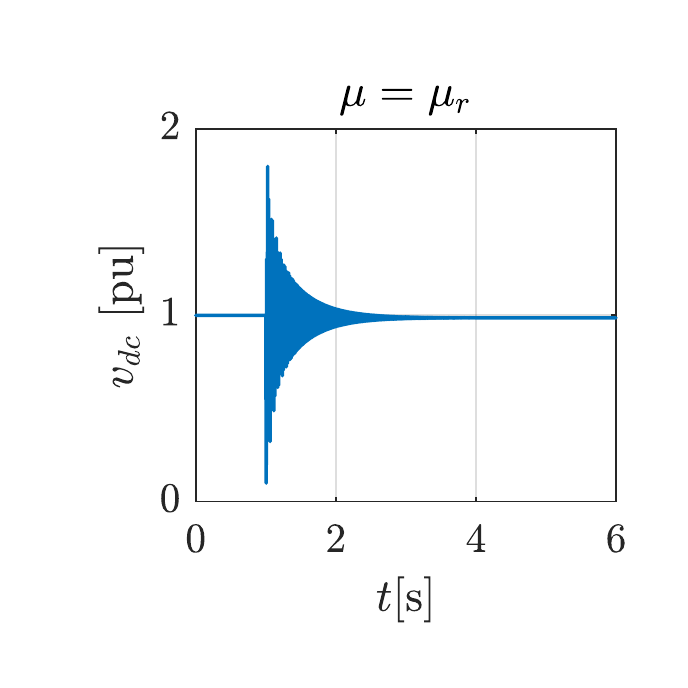}}\hfill
{\includegraphics[trim=6mm 6mm 6mm  6mm ,clip,width=0.18\textwidth]{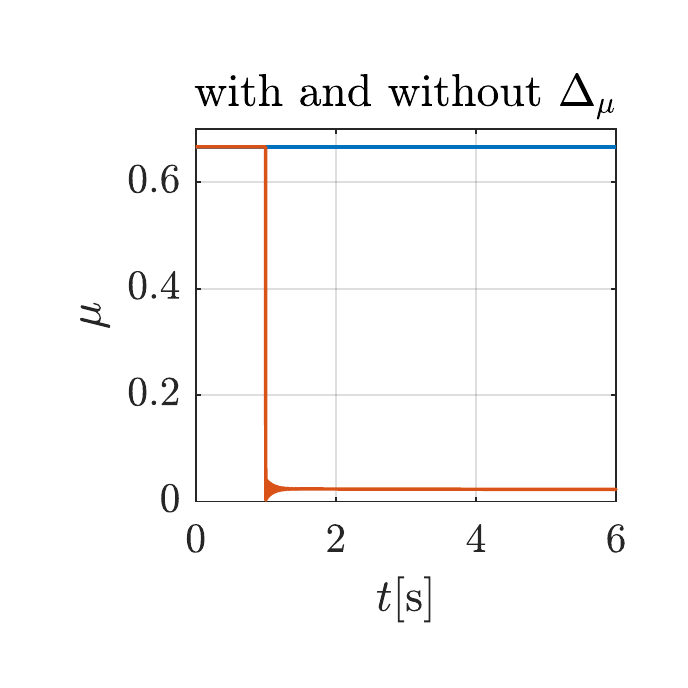}}\hfill
{\includegraphics[trim=6mm 6mm 6mm  6mm ,clip,width=0.18\textwidth]{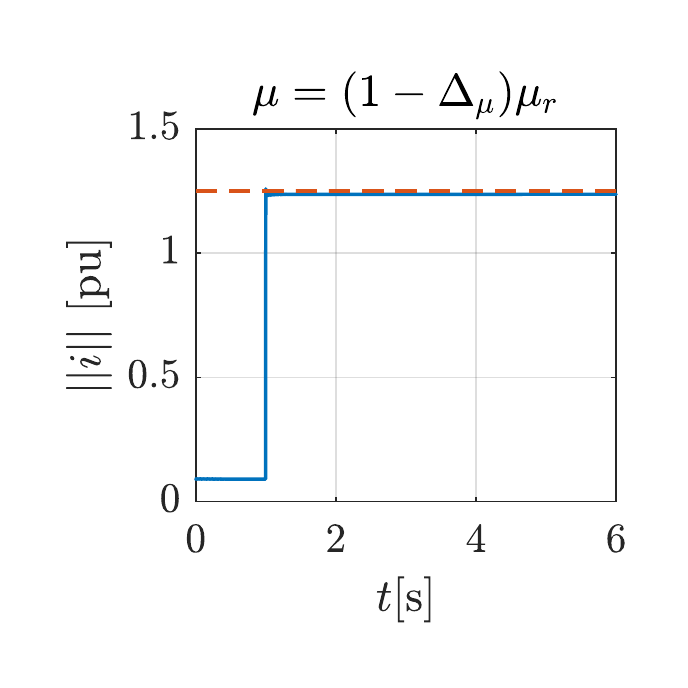}}\hfill
{\includegraphics[trim=6mm 6mm 6mm  6mm ,clip,width=0.18\textwidth]{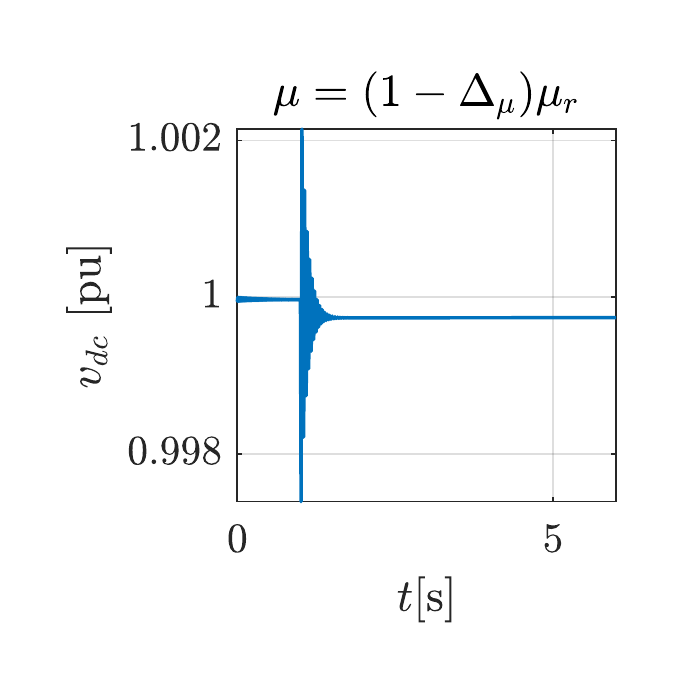}}
\caption{Filter current, dc voltage, and modulation magnitude time-series of \eqref{eqs:sys COI} following a three-phase-to-ground fault with and without current-limiting control \eqref{eq:mu} in per-unit system according to \cite[App. B]{YI10}.}\label{fig:grid fault}
\end{figure*}
\begin{remark}\textup{(Comments on droop slope)}\\ 
With the particular choice of $\eta=\omega_0/v_\tx{dc,r}$ (see Remark~\ref{rem:variants}) the angle term in \eqref{eq:droop slope} simplifies to $\beta_{\theta_\tx{x}}=- \gamma \sin\big( (\theta_\tx{x}  - \theta_\tx{r} ) / 2 \big)$. Furthermore, assuming $\theta_\tx{x}  - \theta_\tx{r}$ is sufficiently small such that $2G_\tx{dc}\beta_{\theta_\tx{x}}/\eta^2$ is negligible compared to the other terms in \eqref{eq:droop slope} yields an approximation of $d_{p-\omega}$ that is 
\begin{equation}
d_{ p-\omega } \approx - \left( \dfrac{ 2 G_\tx{dc} } { \eta^2 } \right) \omega_\tx{x} + \left( \dfrac{ i_0 } { \eta } \right)
\end{equation} 
and coincides with the droop slope of the standard matching control \cite[Prop. 5]{AJD18}. 
We close by remarking that in a multi-converter network, the ratio of droop slopes determines the proportional power-sharing among the converters.
\end{remark}
\vspace{-0.3cm}
\section{Numerical Examples}\label{sec:simulation}
In this section, we consider the closed-loop dynamics of the converter-\ac{coi} system described in the Section \ref{subsec:COI} with parameters in Table \ref{tab:parameters}. The converter parameters are taken from \cite[Table 1]{TGAKD20} (that correspond to a commercial device). In the following, we verify the convergence of closed-loop dynamics \eqref{eqs:sys COI} and the performance of current-limiting control \eqref{eq:mu}. Moreover, we investigate the influence of \ac{hac} on the frequency performance. We remind the reader that the numerical examples in MATLAB/Simulink environment are publicly available  \cite{T20}.    
\begin{table}[b!]
	\caption{The parameters of converter-\ac{coi} system \eqref{eqs:sys COI}. \label{tab:parameters}}
	\centering
	\renewcommand{\arraystretch}{1.4}
	\begin{tabular}[]{c|c|c}
		\hline
		\multicolumn{3}{c}{converter model parameters and nominal values}\\
		\hline
		$S_\tx{r,c}=0.5~\mathrm{[MVA]}$& $v_\tx{r}=816.4~\mathrm{[V]}$ & $\omega_0=2\pi50~\mathrm{[rad/s]}$
		\\ 
		$c_\tx{dc}=0.008~\mathrm{[F]}$  & $\ell=\ell_\tx{g}=200~\mathrm{[\mu H]}$ & $c=300~\mathrm{[\mu F]}$
		\\ 
		$r=r_\tx{g}=0.001~\mathrm{[\Omega]}$ & $g_\tx{dc}=g=0.001~[\Omega^{\mathrm{-1}}]$ & $\tau_\tx{dc}=50 ~\mathrm{[ms]}$
		\\
		\hline
		\multicolumn{3}{c}{center of inertia model parameters}
		\\
		\hline
		$S_\tx{r,g}=5~\mathrm{[MVA]}$&$H=5~\mathrm{[s]}$ & $D=100$
		\\ 
		\hline
		\multicolumn{3}{c}{control parameters and reference values}
		\\
		\hline
		$\theta_\tx{r}=0$ & $i_\tx{dc,r}=0$ & $v_\tx{dc,r}=3v_\tx{r}$ 
		\\
		$\eta=0.01$ & $\gamma=10000$ & $\kappa=2$
		\\
		$\mu_\tx{r}=2v_\tx{r}/v_\tx{dc,r}$ &$b=v_\tx{r}/\omega_0$& $T_\tx{m}=D\omega_0$
		\\
		\hline
	\end{tabular}
\end{table}
\subsection{Convergence of Random Initial Conditions}
Consider the closed-loop dynamics of the converter-\ac{coi} system \eqref{eqs:sys COI} with angle feedback implementation \eqref{eq:u_sw} and parameters in Table \ref{tab:parameters}. 
Figure \ref{fig:timeseries} illustrates the time-evolution of $\theta$ and all per-unit-normalized Euclidean states $\ubar{y_{j}}/\ubar{y_{j}}^\star$ for $j=1,\ldots,9$ starting from three different initial conditions. Note that since $\theta_r=0$, Figure \ref{fig:timeseries} also depicts the converter angle synchronization with that of the COI grid model.
Although depending on the initial value $\ubar{x}(0)$, $\theta$ converges to either $\theta_\tx{r}$ or $\theta_\tx{r}\pm 2\pi$ (see Remark~\ref{rem:switching}), $\ubar{y}^\star$ is unique (due to its $2\pi$-periodicity in angle). 
The fast convergence of angle and dc voltage are underpinned by relatively large $\gamma/\eta$ ratio and $\kappa$. 
In contrast, the oscillatory behavior of the ac states are due to the negligible physical damping $r$, $g$, $r_\tx{g}$ and the fact that their dynamics are influenced by the sluggish \ac{coi} frequency (via the impedance and admittance matrices in \eqref{eqs:sys COI4}-\eqref{eqs:sys COI5}). 
 Moreover, the slow convergence of frequency is because of 
 the fact that -- from practical perspective -- $\omega(0)$ is far from $\omega^\star$ (cf. the convergence timescale of the initial condition with $\omega(0)$ close to $\omega^\star$ (red trajectories) and the evolution of other initial states).
Nonetheless, with a sufficiently large $D$ and $\gamma$ (see stability condition \eqref{eqs:stability conditions COI}) the asymptotic convergence of trajectories is guaranteed under \ac{hac}. 
\begin{figure}[t!]
	\centering
	\begin{tikzpicture}
	\begin{scope}[spy using outlines={gray!75!, magnification=10, size=1.4cm, connect spies,circle,ultra thin}]
	\node[](image) at (0,0){\hspace{-1mm}\includegraphics[trim=5mm 5mm 5mm 10mm ,clip,width=0.95\columnwidth]{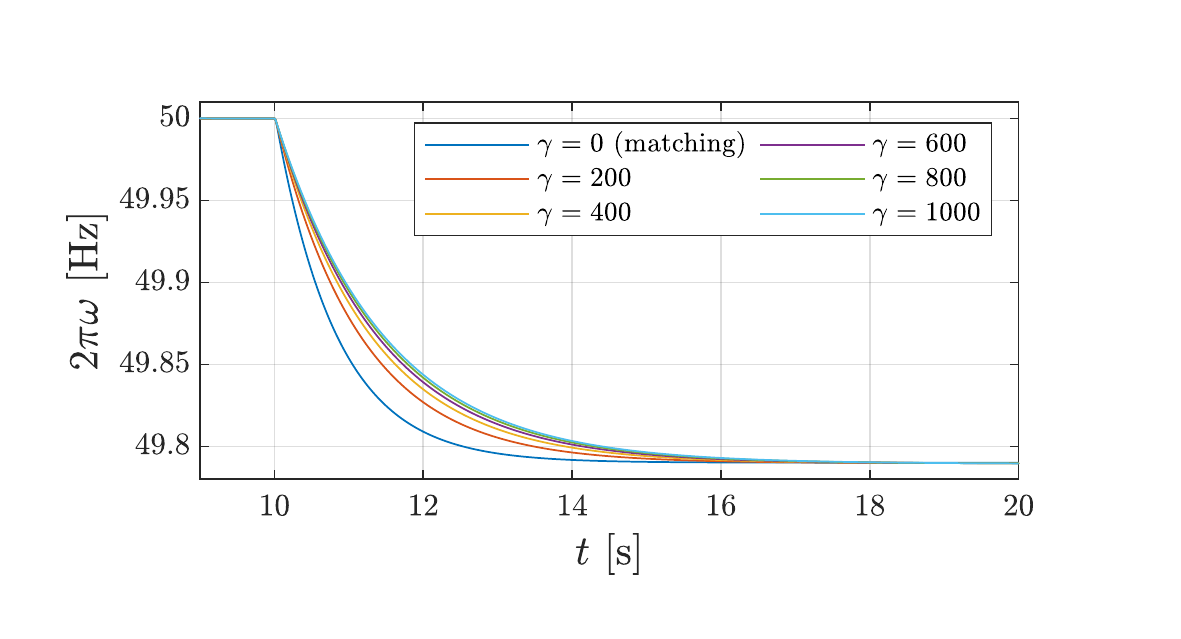}};
	\spy[fill=white] on (-2.43,1.5) in node at (-0.5,0.1);
	\end{scope}
	\end{tikzpicture}
	\caption{Post-contingency frequency evolution of \eqref{eqs:sys COI} for different \ac{hac} tuning.}\label{fig:frequency}
	\includegraphics[trim=5mm 5mm 5mm 5mm ,clip,width=0.95\columnwidth]{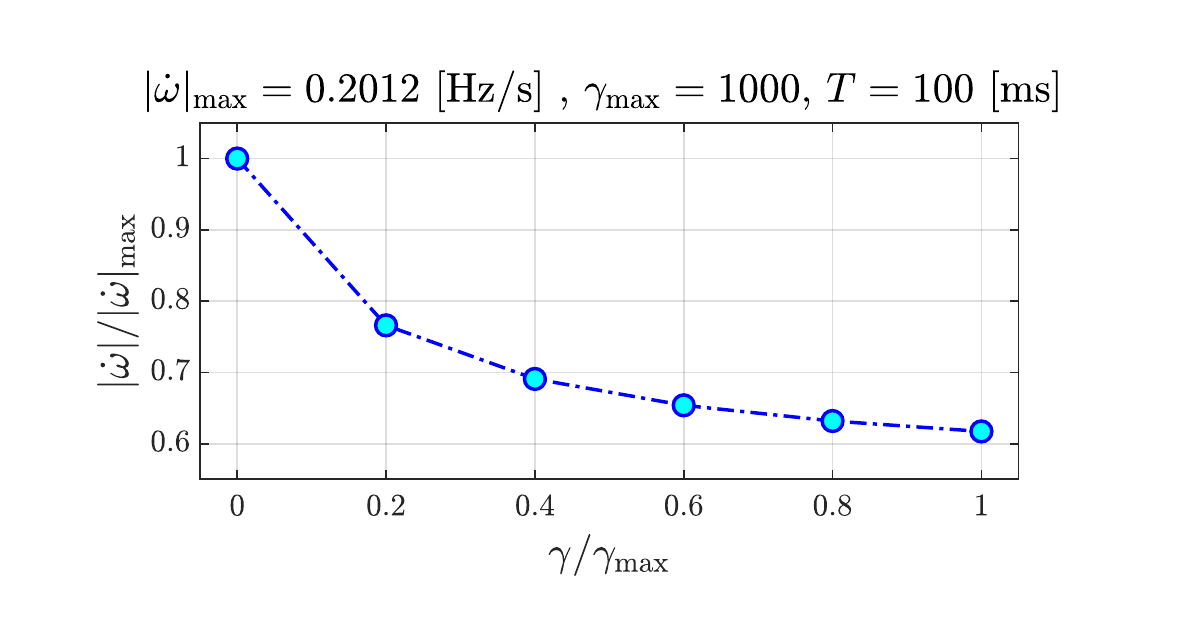}
	\caption{Post-contingency normalized \ac{rocof} for different \ac{hac} tuning corresponding to the frequency time-series in Figure \ref{fig:frequency}.}\label{fig:ROCOF}
\end{figure}
\subsection{Current-Limiting Control Performance}\label{subsec:current-limiting example}
Consider the system in previous example in combination with the current-limiting control \eqref{eq:mu} with $\Delta_\mu$ as in Proposition \ref{prop:current limitation} where $\beta=0.25$, $i_\tx{th}=1.25$ pu, and $D_\tx{min}=0.01$. In what follows, we consider a balanced three-phase-to-ground fault (see \cite[Sec. 3.7]{kundur1994power} for the fault modeling) at the filter capacitance node in Figure \ref{fig:SCIBsys} driving $\n{v}\rightarrow0$ and $\mc{D}\to 0$. Figure \ref{fig:grid fault} shows that the current-limiting control \eqref{eq:mu} (by an immediate reduction of the constant reference $\mu_r$) effectively limits the post-fault current magnitude that aggressively exceeds $i_\tx{th}$ under constant modulation magnitude (that in practice trips the converter protection mechanisms). It is noteworthy that $\Delta_\mu$ also significantly reduces the magnitude of oscillations in dc voltage (cf. to the case with constant modulation magnitude). 
\subsection{\ac{hac} Influence on Frequency Performance}\label{subsec:frequency performance}
In this example, we consider the converter-\ac{coi} closed-loop system \eqref{eqs:sys COI} with $\eta=\omega_0/v_\tx{dc,r}=0.128$ which renders the dc term in \ac{hac} identical to the standard matching control (see Remark \ref{rem:variants} and \cite{AJD18}). Initially, it is assumed that $\gamma=0$. Then, $\kappa=5$ in \eqref{eqs:sys COI2} and $D=300$ \eqref{eqs:sys COI6} are selected such that the converter and \ac{coi} grid model exhibit equal post-contingency load-sharing (i.e., equal increase or decrease in their active power injections). In this example, the contingency is an active power load disturbance that is sized to $0.5 S_\tx{r,c}$ and is modeled by connecting a constant impedance $\m{Z}_\tx{load}$ in parallel connection with the ac-side capacitance in Figure \ref{fig:SCIBsys}. Figure \ref{fig:frequency} illustrates the post-contingency evolution of \ac{coi} frequency (in Hertz) for different $\gamma$ values. Figure \ref{fig:frequency} suggests that \ac{hac} by the means of its angle feedback in \eqref{eq:omega_c}  significantly improves the frequency response of the standard matching control. This improvement is more tangible in the \ac{rocof} performance metric \cite[Sec. III-A]{poolla_placement_2018} i.e.,
\begin{equation}
|{\dot{\omega}}|\coloneqq\frac{|\omega(t_0+T)-\omega(t_0)|}{T},
\end{equation}   
where $t_0$ denotes the time when the disturbance is applied and $T$ denotes the \ac{rocof} approximation time horizon. Figure \ref{fig:ROCOF} shows that the \ac{rocof} exponentially and drastically decreases as $\gamma$ increases. In other word, the angle feedback of the \ac{hac} \eqref{eq:omega_c} serves as a remedy for the suboptimal matching control \ac{rocof} performance (see \cite[Fig. 11]{TGAKD20} for a comparison).  The pure matching control senses the load disturbance (and accordingly modifies the angle dynamics) once its aftermath is propagated to the dc voltage dynamics through filter current dynamics. However, \ac{hac} with its multi-variable nature that includes both dc and ac feedback reacts to the disturbance on a slightly faster timescale that, in our opinion, explains its enhanced frequency response. 
\section{Summary and Outlook}\label{sec:conclusion}
In this paper, we introduced a new grid-forming control termed hybrid angle control (\ac{hac}) \eqref{eq:omega_c}. We formally established the existence, uniqueness, and almost global stability of the closed-loop equilibria under mild parametric conditions. We extended the stability guarantees of \ac{hac} by considering grid dynamics and synthesized a new current-limiting control to account for the converter's safety constraints. Furthermore, a practical implementation of \ac{hac}, its intrinsic power-frequency droop behavior, and a feedforward ac voltage and power control were discussed. Last, the performance of control proposals \eqref{eq:omega_c} and \eqref{eq:delta_mu} was investigated with numerical examples. Our future works will include, 1) stability analysis of interconnected converters under \ac{hac}, 2) exploring the dynamic response of the \ac{hac} and comparing its performance and robustness \ac{wrt} the other control techniques, and 3) investigating the compatibility of proposed current-limiting control with different grid-forming control strategies.
\appendix
\subsection{Proof of Technical Results}\label{App:proofs}
\begin{proof}[Proof of Proposition \ref{prop: instability}]
Consider the shorthand $\dot{x} = f(x)$ for \eqref{eqs:sys} and let $\m{J}_\tx{f}(x)$ be the Jacobian of $f(x)$, then $\det{ \big( \m{J}_\tx{f}(x) \big) } = {\det{ \big( \m{H}(x) \big) }} / { \tau_\tx{dc} c_\tx{dc} ( \ell c \ell_\tx{g} )^2 }$ where $\m{H}(x) = \big( \begin{smallmatrix} \m{H}_{11} & \m{H}_{12} \\ \m{H}_{21} & \m{H}_{22} \end{smallmatrix} \big) \coloneqq $ 
\begin{equation*}\label{eq:J}
\setlength{\extrarowheight}{4pt}
\setlength{\arraycolsep}{4pt}
\left( \begin{array}{cc:cccc}
-\dfrac{ \gamma }{ 2 } \cos\left( \dfrac{ \theta - \theta_\tx{r} } { 2 } \right) & 0 & \eta & 0_2^\top & 0_2^\top & 0_2^\top
\\ 
0 & -1 & {- \kappa} & 0_2^\top & 0_2^\top & 0_2^\top
\\
\hdashline
- \dfrac{ \partial m(\theta) } { \partial \theta}^{\top} i & 1 & -  g_\tx{dc} &  - m(\theta)^\top & 0_2^\top & 0_2^\top
\\
v_\tx{dc} \dfrac{ \partial m(\theta) } { \partial \theta }^\top & 0_2 & m(\theta) &  -\m{Z} &  - \m{I} & {0}_{2\times 2}
\\
0_2 & 0_2 &  0_2 & \m{I} &  - \m{Y} &  - \m{I}
\\
0_2 & 0_2 &  0_2 & 0_{2\times 2} & \m{I} &  - \m{Z}_\tx{g}
\end{array} \right).
	\end{equation*} 
Evaluating $\m{H}_{11}$ at $x^\star_\tx{u}$ in \eqref{eq:EqSet} results in $\m{H}_{11}=
\Big(\begin{smallmatrix} {\gamma}/{2} & 0 \\0 & -1 \end{smallmatrix}\Big)$
which is invertible for $\gamma>0$. Thus, the overall determinant is
\begin{equation}\label{eq:DetJ}
\det(\m{H}(x^\star_\tx{u}))=\det(\m{H}_\tx{11})\det(\m{H}_\tx{22}-\m{H}_\tx{21}\m{H}^{-1}_\tx{11}\m{H}_\tx{12}).
\end{equation}
	Define $\m{K} \coloneqq \m{H}_\tx{22}-\m{H}_\tx{21}\m{H}^{-1}_\tx{11}\m{H}_\tx{12}=$
	\begin{equation*}
	\setlength{\extrarowheight}{4pt}
	\setlength{\arraycolsep}{4pt}
	\left(\begin{array}{cccc}
	- G_\tx{dc}+\left(\dfrac{2\eta }{\gamma}\right)\dfrac{\partial m(\theta)}{\partial \theta}\Big|_{\theta=\theta_\tx{r}}^\top i^\star~~ & - m(\theta_\tx{r})^\top & 0_2^\top & {0}_2^\top\vspace{1mm}
	\\
	{m(\theta_\tx{r})}-\left(\dfrac{2\eta v_\tx{dc}^\star}{\gamma}\right) \dfrac{\partial m(\theta)}{\partial \theta}\Big|_{\theta=\theta_\tx{r}} & -\m{Z} & -\m{I} & {0}_{2\times 2}
	\\
	{0}_2 & \m{I} & -\m{Y} & {-\m{I}}
	\\
	0_2 & 0_{2\times 2} & \m{I} & {-\m{Z}_\tx{g}}
	\end{array}\right)
	\end{equation*}
	where $G_\tx{dc} \coloneqq g_\tx{dc}+\kappa$ and consider the symmetric part of $\m{K}$ i.e., $\m{K}_\tx{S} \coloneqq (1/2)(\m{K}+\m{K}^\top)$.  Next, we show that $\m{K}_\tx{S}\prec0$ under \eqref{eq:stability condition}. Schur complements analysis yields that $\m{K}_\tx{S}\prec0$ iff
	\begin{equation}\label{eq:H_s ND condition}
	\dfrac{2\eta \mu_\tx{r}\gamma\dfrac{\partial\psi(\theta)}{\partial\theta}\Big|_{\theta=\theta_\tx{r}}^\top i^\star}{G_\tx{dc}}+\dfrac{(\eta\mu_\tx{r} v^\star_\tx{dc})^2}{rG_\tx{dc}}<\gamma^2\,.
	\end{equation}
	We apply the identity \eqref{id:binomial} to the first term on the \ac{rhs} of \eqref{eq:H_s ND condition} 
	\begin{align}
	2\mu_\tx{r}\dfrac{\partial\psi(\theta)}{\partial\theta}\Big|_{\theta=\theta_\tx{r}}^\top i^\star&\leq \bigg|\bigg|\dfrac{\partial\psi(\theta)}{\partial\theta}\Big|_{\theta=\theta_\tx{r}}\bigg|\bigg|^2+(\mu_\tx{r}\n{i^\star})^2
	\nonumber
	\\
	&=1+(\mu_\tx{r}\n{i^\star})^2.
	\label{eq:bound i_net}
	\end{align}
	Subsequently, taking into account the bound in \eqref{eq:bound i_net}, if 
	\begin{equation} \label{eq:H_s ND condition2}
	\dfrac{\eta\gamma}{G_\tx{dc}}+\dfrac{\eta\gamma(\mu_\tx{r}\n{i^\star})^2}{G_\tx{dc}}+\dfrac{(\eta\mu_\tx{r} v^\star_\tx{dc})^2}{rG_\tx{dc}}<\gamma^2\,,
	\end{equation}
	then \eqref{eq:H_s ND condition} is satisfied. Further, dividing \eqref{eq:H_s ND condition2} by $\gamma$ results in
	\begin{equation}\label{eq:H_s ND condition3}
	\dfrac{\eta}{\bar{\beta}g_\tx{dc}}+\dfrac{\eta(\mu_\tx{r}\n{i^\star})^2}{\bar{\beta}g_\tx{dc}}+\dfrac{\eta(\mu_\tx{r} v^\star_\tx{dc})^2}{ \bar{\alpha} r}<\gamma
	\end{equation}
	where $\bar{\alpha} \coloneqq {\gamma G_\tx{dc}}/{\eta}$ and $\bar{\beta}\coloneqq G_\tx{dc}/g_\tx{dc}$. Since $\bar{\beta}>1$ by definition (recall that $G_\tx{dc}=\kappa+g_\tx{dc}$), if $\bar{\alpha}>1$ then the \ac{lhs} of \eqref{eq:H_s ND condition3} is strictly smaller that the \ac{lhs} of \eqref{eq:stability condition}. That means if $\bar{\alpha}>1$ then \eqref{eq:stability condition} implies \eqref{eq:H_s ND condition3}. To show that $\bar{\alpha}>1$  that equals $\gamma>{\eta}/\bar{\beta}g_\tx{dc}$ consider that if \eqref{eq:stability condition} holds then $\gamma>{\eta}/g_\tx{dc}>{\eta}/\bar{\beta}g_\tx{dc}$ hence $\bar{\alpha}>1$. 
	
To sum up, under \eqref{eq:stability condition}, $\m{K}_\tx{S}\prec0$. Thus, $\m{K}$ has all eigenvalues in the open left half-plane. Since $\mathrm{dim}(\m{K})=7$, then $\det(\m{K})<0$ and by \eqref{eq:DetJ} $\det(\m{H}(x^\star_\tx{u}))=-\gamma\det(\m{K})/2>0$ which means $\det(\m{J}_\tx{f}(x^\star_\tx{u}))>0$.  Since $\mathrm{dim}(\m{J}_\tx{f}(x))=9$, then $\m{J}_\tx{f}(x^\star_\tx{u})$ has at least one positive real eigenvalue. Instability of $x^\star_\tx{u}$ follows from Lyapunov's indirect method \cite[Th. 4.7]{khalil_nonlinear_2002} and its global inset has zero Lebesgue measure invoking \cite[Prop. 11]{monzon_local_2006}.
\end{proof}
\begin{proof}[Proof of Proposition \ref{prop:consistent references}]
	The power injection to the \ac{ib} at equilibrium \cite[Def. 2]{CGBF19} can be expressed as
	\begin{subequations}\label{eqs:p*_g,q*_g}
		\begin{align}
		&p^\star_\tx{g}  = - \dfrac
		{ \n{ v^\star } v_\tx{r} \big( r_\tx{g} \cos\left( \delta^\star_{\tx{b}v} \right) + \ell_\tx{g} \omega_0 \sin\left( \delta^\star_{\tx{b}v} \right) \big) }
		{ r_\tx{g}^2 + \left( \ell_\tx{g} \omega_0 \right)^2 },
		\\
		&q^\star_\tx{g}  = - \dfrac
		{ \n{ v^\star } v_\tx{r} \big( \ell_\tx{g} \omega_0 \cos\left( \delta^\star_{\tx{b}v} \right) - r_\tx{g} \sin\left( \delta^\star_{\tx{b}v} \right) \big) }
		{ r_\tx{g}^2 + \left( \ell_\tx{g} \omega_0 \right)^2 },	
		\end{align}
	\end{subequations}
	where $\delta^\star_{\tx{b}v} \coloneqq \theta^\star_\tx{b} - \theta^\star_v$ and $\theta^\star_v=\tan^{-1}(v^\star_\tx{q}/v^\star_\tx{d})$. Taking into account the line loss and the power associated with the shunt element, we can compute $p^\star_\tx{f}$ and $q^\star_\tx{f}$ (see Figure \ref{fig:SCIBsys}) by
	\begin{align*}
	p^\star_\tx{f} & = p^\star_\tx{g} + \bigg( \dfrac{ r_\tx{g} }{ r_\tx{g}^2 + \left( \ell_\tx{g} \omega_0 \right)^2 } + g \bigg)\n{ v^\star }^2 ,
	\\
	q^\star_\tx{f} & = q^\star_\tx{g} + \bigg( \dfrac { \ell_\tx{g} \omega_0 } { r_\tx{g}^2 + \left( \ell_\tx{g} \omega_0 \right)^2 } - c \omega_0 \bigg)\n{ v^\star }^2 .
	\end{align*}
	Note that $p^\star_\tx{f}$ and $q^\star_\tx{f}$ are also expressed by
	\begin{subequations}\label{eqs:p*_f,q*_f}
		\begin{align}
		p^\star_\tx{f} & = -\dfrac
		{ \n{v^\star_\tx{s} } \n{ v^\star } \big( r \cos\left( \delta_{v\tx{c}} \right) + \ell \omega_0 \sin\left( \delta_{v\tx{c}} \right) \big) } 
		{ r^2 + \left( \ell \omega_0 \right)^2 }, 
		\\
		q^\star_\tx{f} & = -\dfrac
		{ \n{ v^\star_\tx{s} } \n{ v^\star } \big( \ell \omega_0 \cos\left( \delta_{v\tx{c}} \right) - r \sin\left( \delta_{v\tx{c}} \right) \big) }
		{ r^2 + \left( \ell \omega_0 \right)^2 },
		\end{align}
	\end{subequations}
	where $\delta_{v\tx{c}} \coloneqq \theta^\star_v - \theta^\star_\tx{c}$ and $\n{v^\star_\tx{s}} \coloneqq \mu^\star v^\star_\tx{dc}$ denotes the equilibrium voltage magnitude before the filter inductor. 
	
	Consider the shorthand $\det(\m{Z}_\tx{g})=r_\tx{g}^2 + ( \ell_\tx{g} \omega_0 )^2 $, and let us define $\vartheta_\tx{g} \coloneqq \tan^{-1}( \ell_\tx{g} \omega_0 / r_\tx{g} )$, $ \sin( \vartheta_\tx{g} )  \coloneqq  { \ell_\tx{g} \omega_0 } / { \sqrt{ \det(\m{Z}_\tx{g})} }$, and $ \cos( \vartheta_\tx{g} )  \coloneqq  { r_\tx{g} } / { \sqrt{ \det(\m{Z}_\tx{g}) } }$. Then, \eqref{eqs:p*_g,q*_g} is equivalent to
	\begin{equation*}
	s^\star_\tx{g}=-\n{ s^\star_\tx{g} } \m{R}( \vartheta_\tx{g} ) \psi( \theta^\star_v - \theta^\star_\tx{b} )
	=-\n{ s^\star_\tx{g} } \psi( \theta^\star_v + \vartheta_\tx{g} - \theta^\star_\tx{b} ),
	\end{equation*}
	where $\n{s^\star_\tx{g}}={ v_\tx{r} \n{ v^\star } } / { \sqrt{ \det(\m{Z}_\tx{g}) } }$ and subsequently, $\hat{s}^\star_\tx{g}=-\psi(\theta^\star_v + \vartheta_\tx{g}-\theta^\star_\tx{b} )$. Similarly, define $\vartheta_\tx{f}  \coloneqq  \tan^{-1}(\ell\omega_0/r)$, then \eqref{eqs:p*_f,q*_f} is equivalent to 
	\begin{equation*}
	s^\star_\tx{f}=-\n{ s^\star_\tx{f} } \m{R}( \vartheta ) \psi( \theta^\star_\tx{c} - \theta^\star_v )
	=-\n{ s^\star_\tx{f} } \psi( \theta^\star_\tx{c} + \vartheta - \theta^\star_v ),
	\end{equation*}
	where $\n{s^\star_\tx{f}}={ \n{ v^\star_\tx{s} } \n{ v^\star } } / { \sqrt{ \det(\m{Z}) } }$. Thus, $\hat{s}^\star_\tx{f}=-\psi( \theta^\star_\tx{c} + \vartheta - \theta^\star_v )$. By the means of \eqref{id: sine angle sum} and \eqref{id: cosine angle sum}
	\begin{align*}
	\Big( 
	\hat{s}_\tx{g}^{\star\top} \left(\begin{smallmatrix} +1 & 0 \\ 0 & -1 \end{smallmatrix}\right) \hat{s}^\star_\tx{f}
	, 
	\hat{s}_\tx{g}^{\star\top} \left(\begin{smallmatrix} 0 & +1 \\+1 & 0 \end{smallmatrix}\right) \hat{s}^\star_\tx{f} 
	\Big)
	=
	\psi( \theta^\star_\tx{c} - \theta^\star_\tx{b}+\delta )
	\end{align*}
	and subsequently, $\m{R}(\delta)^\top\psi( \theta^\star_\tx{c} - \theta^\star_\tx{b}+\delta )=\psi(\theta^\star)$. 
	
	Thus, $\psi(\theta_\tx{r})$ that is uniquely defined by \eqref{eq:psi(theta_r)} coincides with the solution of power flow equations, i.e., $\psi(\theta_\tx{r})=\psi(\theta^\star)$. To prove the second statement \eqref{eq:mu_r}, $\mu^\star$ is derived from the expression of $\n{s^\star_\tx{f}}$ i.e., 
	$\mu^\star={\n{s^\star_\tx{f}} \sqrt{ \det(\m{Z})}}/{{ v^\star_\tx{dc}\n{v^\star} }}
	$, 	which shows $\mu_\tx{r}$ defined by \eqref{eq:mu_r} is consistent with $\mu^\star$. 
\end{proof}
\begin{proof}[Proof of Proposition \ref{prop:droop slope}]
	Note that by the relative angle dynamics \eqref{eqs:sys COI1} at equilibrium $\omega_\tx{c,x}=\omega_\tx{x}$.
	Multiply \eqref{eqs:sys COI3} at equilibrium by $v_\tx{dc,x}$
	\begin{equation*}
	i_\tx{dc,x}v_\tx{dc,x}-g_\tx{dc}v^2_\tx{dc,x}-p_\tx{net,x}=0,
	\end{equation*} 
	and replace $v_\tx{dc,x}$ with the expression from $\eqref{eq:omega_c}$, that is, $v_\tx{dc,x}=(\omega_\tx{x}-\beta_{\theta_\tx{x}})/\eta$ which results in
	\begin{equation}\label{eq:p_x}
	p_\tx{net,x}=\dfrac{i_\tx{dc,x}(\omega_\tx{x}-\beta_{\theta_\tx{x}})}{\eta}-\dfrac{g_\tx{dc}(\omega_\tx{x}-\beta_{\theta_\tx{x}})^2}{\eta^2}.
	\end{equation} 
	Replacing $i_\tx{dc,x}$ from \eqref{eqs:sys1+} at equilibrium results in
	\begin{equation}\label{eq:p_x2}
	p_\tx{net,x}=\dfrac{i_\tx{0}(\omega_\tx{x}-\beta_{\theta_\tx{x}})}{\eta}-\dfrac{(\kappa+g_\tx{dc})(\omega_\tx{x}-\beta_{\theta_\tx{x}})^2}{\eta^2}.
	\end{equation} 
	Hence, \eqref{eq:droop slope} directly follows by linearizing \eqref{eq:p_x2} \ac{wrt} $\omega_\tx{x}$. 
\end{proof}
\subsection{Coordinate Transformations and Identities}
\subsubsection{$\alpha\beta$-coordinates}\label{app:Clarke}
for a three-phase quantity $z_\tx{abc} \coloneqq (z_\tx{a},z_\tx{b},z_\tx{c})\in\mathbb{R}^3$ that is balanced i.e., $z_\tx{a}+z_\tx{b}+z_\tx{c}=0$ the magnitude preserving Clarke transformation is defined by
\begin{equation}
\setlength{\extrarowheight}{10pt}
z_{\alpha\beta}=(z_\alpha,z_\beta) \coloneqq \m{C}z_\tx{abc}=\dfrac{2}{3}\begin{pmatrix}
1&-\dfrac{1}{2}&-\dfrac{1}{2}\\
0&\dfrac{\sqrt{3}}{2}&-\dfrac{\sqrt{3}}{2}
\end{pmatrix}z_\tx{abc}\,.
\end{equation}
\subsubsection{Polar coordinates}\label{app:Polar}
the transformation from Cartesian to polar coordinates i.e.,  $\mathcal{P}:\mathbb{R}^2\backslash\{0\}\rightarrow\mathbb{R}_{>0}\times\mathbb{S}^1$ is
\begin{equation}\label{eqs:dq2polar}
\big(\n{z},\theta_z\big)=\mathcal{P}(z) \coloneqq \left(\sqrt{z^2_1+z^2_2}\,,\tan^{-1}\dfrac{z_2}{z_1}\right).
\end{equation}
Moreover, the inverse transformation is given by
\begin{equation}\label{eqs:polar2dq}
(z_1,z_2)=\mc{P}^{-1}\big(\n{z},\theta_z\big) \coloneqq \n{z}\psi(\theta_z).
\end{equation}  
Note that the polar coordinates are well-defined for the entire Cartesian space except the origin since $\mathcal{P}(0)$ is not unique.
\begin{lemma}\textup{(Algebraic and trigonometric identities)}\\
	For $ a , b \in \mathbb{R}^2 $, $ \epsilon \in \mathbb{R}_{>0} $ and $ \varphi , \phi \in \mathbb{S}^1$ the followings hold
	\begin{align}
	\pm a^\top b & \leq \epsilon^2 \n{a}^2 + \dfrac{1} {4\epsilon^2} \n{b}^2, \label{id:binomial}\\
	\sin^2\dfrac{\varphi}  {2} & = \big({ 1 - \cos\varphi }\big) / {2}, \label{id: sine half angle}\\
	\cos^2\dfrac{\varphi}  {2} & = \big({ 1 + \cos\varphi }\big) / {2}, \label{id: cosine half angle}\\
	\sin( \varphi \pm \phi ) & = \sin(\varphi) \cos(\phi) \pm \cos(\varphi) \sin(\phi), \label{id: sine angle sum}\\
	\cos(\varphi \pm \phi ) & = \cos(\varphi) \cos(\phi) \mp \sin(\varphi) \sin(\phi).\label{id: cosine angle sum}
	\end{align}  	
\end{lemma}
\bibliographystyle{IEEEtran}
\bibliography{IEEEabrv,Ref}

\begin{thebibliography}{10}
\providecommand{\url}[1]{#1}
\csname url@samestyle\endcsname
\providecommand{\newblock}{\relax}
\providecommand{\bibinfo}[2]{#2}
\providecommand{\BIBentrySTDinterwordspacing}{\spaceskip=0pt\relax}
\providecommand{\BIBentryALTinterwordstretchfactor}{4}
\providecommand{\BIBentryALTinterwordspacing}{\spaceskip=\fontdimen2\font plus
\BIBentryALTinterwordstretchfactor\fontdimen3\font minus
  \fontdimen4\font\relax}
\providecommand{\BIBforeignlanguage}[2]{{%
\expandafter\ifx\csname l@#1\endcsname\relax
\typeout{** WARNING: IEEEtran.bst: No hyphenation pattern has been}%
\typeout{** loaded for the language `#1'. Using the pattern for}%
\typeout{** the default language instead.}%
\else
\language=\csname l@#1\endcsname
\fi
#2}}
\providecommand{\BIBdecl}{\relax}
\BIBdecl

\bibitem{MDHHV18}
F.~Milano, F.~D\"{o}rfler, G.~Hug, D.~J. Hill, and G.~Verbič, ``Foundations
  and challenges of low-inertia systems,'' in \emph{{Power} {Systems}
  {Computation} {Conference} ({PSCC})}, 2018.

\bibitem{TGAKD20}
A.~{Tayyebi}, D.~{Groß}, A.~{Anta}, F.~{Kupzog}, and F.~{Dörfler},
  ``Frequency stability of synchronous machines and grid-forming power
  converters,'' \emph{{IEEE} Trans. Emerg. Sel. Topics Power Electron.},
  vol.~8, no.~2, pp. 1004--1018, 2020.

\bibitem{MOVAH19}
U.~Markovic, O.~Stanojev, E.~Vrettos, P.~Aristidou, and G.~Hug, ``Understanding
  stability of low-inertia systems,'' 2019, {Preprint} available at
  {http://engrxiv.org/jwzrq}.

\bibitem{QQYYB:19}
Q.~Peng, Q.~Jiang, Y.~Yang, T.~Liu, H.~Wang, and F.~Blaabjerg, ``On the
  stability of power electronics-dominated systems: challenges and potential
  solutions,'' \emph{{IEEE} Trans. Ind. Appl.}, vol.~55, no.~6, pp. 7657--7670,
  2019.

\bibitem{CTGAKF19}
A.~Crivellaro, A.~Tayyebi, C.~Gavriluta, D.~Gro{\ss}, A.~Anta, F.~Kupzog, and
  F.~D{\"o}rfler, ``Beyond low-inertia systems: Massive integration of
  grid-forming power converters in transmission grids,'' in \emph{{IEEE} {PES}
  {General} {Meeting}}, 2020, {To appear, {preprint} available at
  https://arxiv.org/abs/1911.02870}.

\bibitem{fang2018inertia}
J.~Fang, H.~Li, Y.~Tang, and F.~Blaabjerg, ``On the inertia of future
  more-electronics power systems,'' \emph{{IEEE} Trans. Emerg. Sel. Topics
  Power Electron.}, vol.~7, no.~4, pp. 2130--2146, 2018.

\bibitem{poolla_placement_2018}
B.~K. {Poolla}, D.~{Groß}, and F.~{Dörfler}, ``Placement and implementation
  of grid-forming and grid-following virtual inertia and fast frequency
  response,'' \emph{{IEEE} Trans. Power Syst.}, vol.~34, no.~4, pp. 3035--3046,
  2019.

\bibitem{TDKZH18}
A.~Tayyebi, F.~D\"{o}rfler, F.~Kupzog, Z.~Miletic, and W.~Hribernik,
  ``Grid-forming converters -- inevitability, control strategies and challenges
  in future grid applications,'' in \emph{{CIRED} {Workshop}}, 2018.

\bibitem{CDA93}
M.~Chandorkar, D.~Divan, and R.~Adapa, ``Control of parallel connected
  inverters in standalone {AC} supply systems,'' \emph{{IEEE} Trans. Ind.
  Appl.}, vol.~29, no.~1, pp. 136--143, 1993.

\bibitem{SDB13}
J.~W. Simpson-Porco, F.~D\"{o}rfler, and F.~Bullo, ``Synchronization and power
  sharing for droop-controlled inverters in islanded microgrids,''
  \emph{Automatica}, vol.~49, no.~9, pp. 2603--2611, 2013.

\bibitem{ZW11}
Q.~C. Zhong and G.~Weiss, ``Synchronverters: {Inverters} that mimic synchronous
  generators,'' \emph{{IEEE} Trans. Ind. Electron.}, vol.~58, no.~4, pp.
  1259--1267, 2011.

\bibitem{d2013virtual}
S.~D’Arco, J.~A. Suul, and O.~B. Fosso, ``A virtual synchronous machine
  implementation for distributed control of power converters in smart grids,''
  \emph{Electric Power Systems Research}, vol. 122, pp. 180--197, 2015.

\bibitem{cvetkovic_modeling_2015}
I.~Cvetkovic, D.~Boroyevich, R.~Burgos, C.~Li, and P.~Mattavelli, ``Modeling
  and control of grid-connected voltage-source converters emulating isotropic
  and anisotropic synchronous machines,'' in \emph{IEEE Workshop on Control and
  Modeling for Power Electronics (COMPEL)}, 2015.

\bibitem{huang2017virtual}
L.~{Huang}, H.~{Xin}, Z.~{Wang}, K.~{Wu}, H.~{Wang}, J.~{Hu}, and C.~{Lu}, ``A
  virtual synchronous control for voltage-source converters utilizing dynamics
  of dc-link capacitor to realize self-synchronization,'' \emph{{IEEE} Trans.
  Emerg. Sel. Topics Power Electron.}, vol.~5, no.~4, pp. 1565--1577, 2017.

\bibitem{CGD17}
S.~Curi, D.~Gro\ss, and F.~D\"{o}rfler, ``Control of low-inertia power grids:
  {A} model reduction approach,'' in \emph{{IEEE} {Conference} on {Decision}
  and {Control} ({CDC})}, 2017.

\bibitem{AJD18}
C.~Arghir, T.~Jouini, and F.~Dörfler, ``Grid-forming control for power
  converters based on matching of synchronous machines,'' \emph{Automatica},
  vol.~95, pp. 273--282, 2018.

\bibitem{AF20}
C.~{Arghir} and F.~{Dörfler}, ``The electronic realization of synchronous
  machines: Model matching, angle tracking, and energy shaping techniques,''
  \emph{{IEEE} Trans. Power Electron.}, vol.~35, no.~4, pp. 4398--4410, 2020.

\bibitem{JMAFD:15}
B.~B. Johnson, M.~Sinha, N.~G. Ainsworth, F.~D{\"o}rfler, and S.~V. Dhople,
  ``Synthesizing virtual oscillators to control islanded inverters,''
  \emph{{IEEE} Trans. Power Electron.}, vol.~31, no.~8, pp. 6002--6015, 2015.

\bibitem{SDJD17}
M.~Sinha, F.~D\"{o}rfler, B.~B. Johnson, and S.~V. Dhople, ``Uncovering droop
  control laws embedded within the nonlinear dynamics of van der pol
  oscillators,'' \emph{{IEEE} Trans. Control Netw. Syst.}, vol.~4, no.~2, pp.
  347--358, 2017.

\bibitem{CGBF19}
M.~{Colombino}, D.~{Groß}, J.~{Brouillon}, and F.~{Dörfler}, ``Global phase
  and magnitude synchronization of coupled oscillators with application to the
  control of grid-forming power inverters,'' \emph{{IEEE} Trans. Autom.
  Control}, vol.~64, no.~11, pp. 4496--4511, 2019.

\bibitem{GCBD19}
D.~{Groß}, M.~{Colombino}, J.~{Brouillon}, and F.~{Dörfler}, ``The effect of
  transmission-line dynamics on grid-forming dispatchable virtual oscillator
  control,'' \emph{{IEEE} Trans. Control Netw. Syst.}, vol.~6, no.~3, pp.
  1148--1160, 2019.

\bibitem{yu2020comparative}
H.~Yu, M.~Awal, H.~Tu, I.~Husain, and S.~Lukic, ``Comparative transient
  stability assessment of droop and dispatchable virtual oscillator controlled
  grid-connected inverters,'' \emph{{IEEE} Trans. Power Electron.}, 2020.

\bibitem{OVME02}
R.~Ortega, A.~van~der Schaft, B.~Maschke, and G.~Escobar, ``Interconnection and
  damping assignment passivity-based control of port-controlled {Hamiltonian}
  systems,'' \emph{Automatica}, vol.~38, no.~4, pp. 585--596, 2002.

\bibitem{sarras2012asymptotic}
I.~Sarras, R.~Ortega, and E.~Panteley, ``Asymptotic stabilization of nonlinear
  systems via sign-indefinite damping injection,'' in \emph{{IEEE} {C}onference
  on {D}ecision and {C}ontrol (CDC)}, 2012.

\bibitem{BSEO17}
N.~Barabanov, J.~Schiffer, R.~Ortega, and D.~Efimov, ``Conditions for {Almost}
  {Global} {Attractivity} of a {Synchronous} {Generator} {Connected} to an
  {Infinite} {Bus},'' \emph{{IEEE} Trans. Autom. Control}, vol.~62, pp.
  4905--4916, 2017.

\bibitem{CT14}
S.~Y. Caliskan and P.~Tabuada, ``Compositional {Transient} {Stability}
  {Analysis} of {Multimachine} {Power} {Networks},'' \emph{{IEEE} Trans.
  Control Netw. Syst.}, vol.~1, no.~1, pp. 4--14, 2014.

\bibitem{T20}
A.~{Tayyebi}, ``{H}ybrid{A}ngle{C}ontrol{(HAC)}: Implementation of grid-forming
  hybrid angle control,'' Git repository, 2020,
  https://github.com/ATayebi/HybridAngleControl-HAC.

\bibitem{bhat2000topological}
S.~P. Bhat and D.~S. Bernstein, ``A topological obstruction to continuous
  global stabilization of rotational motion and the unwinding phenomenon,''
  \emph{Systems \& Control Letters}, vol.~39, no.~1, pp. 63--70, 2000.

\bibitem{YI10}
A.~Yazdani and R.~Iravani, \emph{Voltage-sourced converters in power systems:
  modeling, control, and applications}.\hskip 1em plus 0.5em minus 0.4em\relax
  John Wiley \& Sons, 2010.

\bibitem{MO19}
F.~Milano and {\'A}.~O. Manjavacas, \emph{Converter-Interfaced Energy Storage
  Systems: Context, Modelling and Dynamic Analysis}.\hskip 1em plus 0.5em minus
  0.4em\relax Cambridge University Press, 2019.

\bibitem{seo2019dispatchable}
G.-S. Seo, M.~Colombino, I.~Subotic, B.~Johnson, D.~Gro{\ss}, and
  F.~D{\"o}rfler, ``Dispatchable virtual oscillator control for decentralized
  inverter-dominated power systems: Analysis and experiments,'' in \emph{IEEE
  Applied Power Electronics Conference and Exposition (APEC)}, 2019.

\bibitem{khalil_nonlinear_2002}
H.~Khalil, \emph{Nonlinear {Systems}}.\hskip 1em plus 0.5em minus 0.4em\relax
  Prentice Hall, 2002.

\bibitem{PD15}
A.~D. {Paquette} and D.~M. {Divan}, ``Virtual impedance current limiting for
  inverters in microgrids with synchronous generators,'' \emph{{IEEE} Trans.
  Ind. Appl.}, vol.~51, no.~2, pp. 1630--1638, 2015.

\bibitem{SGCD19}
I.~Subotić, D.~Groß, M.~Colombino, and F.~Dörfler, ``A {L}yapunov framework
  for nested dynamical systems on multiple time scales with application to
  converter-based power systems,'' 2019, {Preprint} available at
  {https://arxiv.org/abs/1911.08945}.

\bibitem{baros2020stability}
S.~Baros, C.~N. Hadjicostis, and F.~O'Sullivan, ``Stability analysis of
  droop-controlled inverter-based power grids via timescale separation,'' 2020,
  {Preprint} available at {https://arxiv.org/abs/2003.11934}.

\bibitem{C13}
J.~H. Chow, Ed., \emph{Power {System} {Coherency} and {Model}
  {Reduction}}.\hskip 1em plus 0.5em minus 0.4em\relax Springer, 2013, vol.~94.

\bibitem{RDB13}
D.~{Romeres}, F.~{Dörfler}, and F.~{Bullo}, ``Novel results on slow coherency
  in consensus and power networks,'' in \emph{European Control Conference
  (ECC)}, 2013.

\bibitem{SP98}
P.~W. Sauer and M.~A. Pai, \emph{Power system dynamics and stability}.\hskip
  1em plus 0.5em minus 0.4em\relax Prentice hall, 1998.

\bibitem{UBA14}
A.~Ulbig, T.~S. Borsche, and G.~Andersson, ``Impact of {Low} {Rotational}
  {Inertia} on {Power} {System} {Stability} and {Operation},'' \emph{IFAC
  Proceedings Volumes}, vol.~47, no.~3, pp. 7290--7297, 2014.

\bibitem{machowski2020power}
J.~Machowski, Z.~Lubosny, J.~W. Bialek, and J.~R. Bumby, \emph{Power system
  dynamics: stability and control}.\hskip 1em plus 0.5em minus 0.4em\relax John
  Wiley \& Sons, 2020.

\bibitem{sadeghkhani2016current}
I.~Sadeghkhani, M.~E.~H. Golshan, J.~M. Guerrero, and A.~Mehrizi-Sani, ``A
  current limiting strategy to improve fault ride-through of inverter
  interfaced autonomous microgrids,'' \emph{{IEEE} Trans. Smart Grid}, vol.~8,
  no.~5, pp. 2138--2148, 2016.

\bibitem{GD19}
D.~{Groß} and F.~{Dörfler}, ``Projected grid-forming control for
  current-limiting of power converters,'' in \emph{Allerton Conference on
  Communication, Control, and Computing}, 2019.

\bibitem{TWDF20}
M.~G. {Taul}, X.~{Wang}, P.~{Davari}, and F.~{Blaabjerg}, ``Current limiting
  control with enhanced dynamics of grid-forming converters during fault
  conditions,'' \emph{{IEEE} Trans. Emerg. Sel. Topics Power Electron.},
  vol.~8, no.~2, pp. 1062--1073, 2020.

\bibitem{B99}
F.~Blanchini, ``Set invariance in control,'' \emph{Automatica}, vol.~35,
  no.~11, pp. 1747 -- 1767, 1999.

\bibitem{kundur1994power}
P.~Kundur, \emph{Power system stability and control}.\hskip 1em plus 0.5em
  minus 0.4em\relax McGraw-hill, 1994.

\bibitem{monzon_local_2006}
P.~Monzon and R.~Potrie, ``Local and global aspects of almost global
  stability,'' in \emph{{IEEE} {Conference} on {Decision} and {Control}
  ({CDC})}, 2006.

\end{thebibliography}

\ifCLASSOPTIONcaptionsoff
  \newpage
\fi
\clearpage
\begin{IEEEbiography}[{\includegraphics[width=1in,height=1.25in,clip,keepaspectratio]{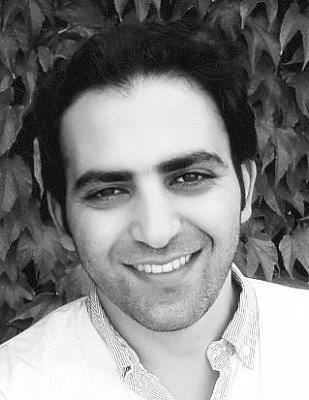}}]{Ali Tayyebi} received his BSc degree in electrical engineering from the University of Tehran, Iran in 2012. In 2014 he received his MSc degree in engineering mathematics (joint MATHMODS program) from University of L'Aquila, and University of Hamburg in Italy and Germany respectively. In 2016, he received his second MSc degree in sustainable transportation and electric power systems  (joint STEPS program) from La Sapienza, University of Nottingham and University of Oviedo respectively in Italy, UK and Spain. From 2014 to 2016, he was the recipient of EU scholarship for master studies. In 2016, he joined Austrian Institute of Technology (AIT) in Vienna, Austria as a master thesis candidate and afterward continued with AIT as research assistant. In 2017, he started his joint PhD project at AIT and Automatic Control Laboratory, Swiss Federal Institute of Technology (ETH) Zürich, Switzerland. 
His main research interest is the non-linear systems and control theory with applications to power system. In particular, his PhD research focuses on the design of grid-forming converter control for low-inertia power system. He has recently won the IEEE PES General Meeting 2020 best paper award.   
\end{IEEEbiography}

\begin{IEEEbiography}[{\includegraphics[width=1in,height=1.25in,clip,keepaspectratio]{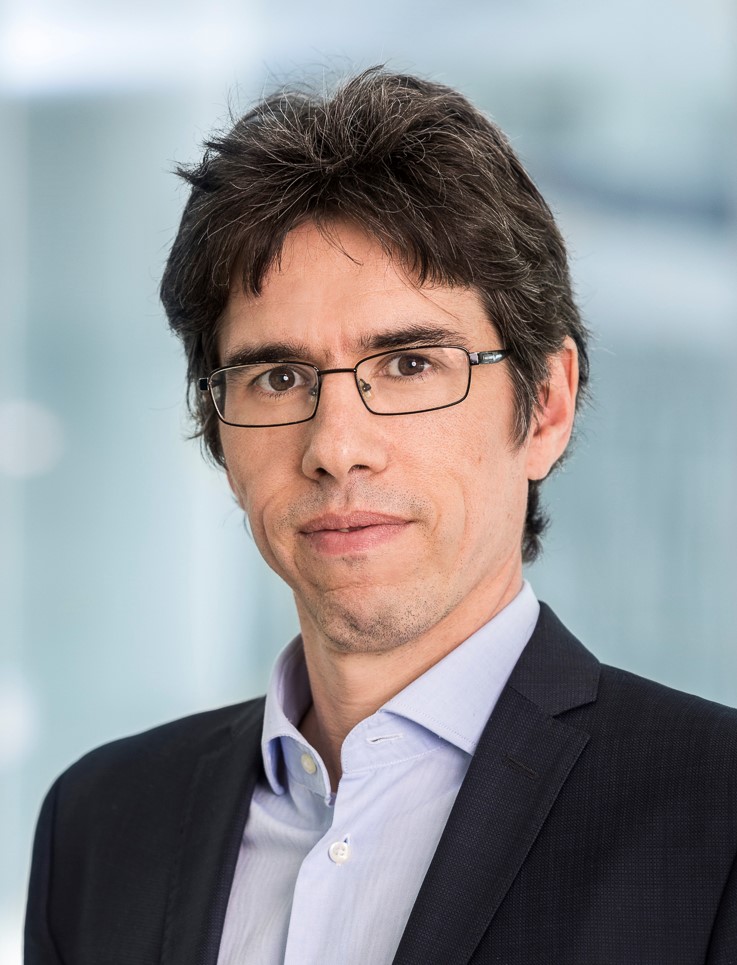}}]{Adolfo Anta}
	received the Licenciatura degree from ICAI Engineering School, Madrid, Spain, in 2002, and the M.Sc. and Ph.D. degrees from the University of California, Los Angeles, CA, USA, in 2007 and 2010, respectively. From 2002 to 2005, he was a Design Engineer with EADS-Astrium and, from 2010 to 2012, he was a Postdoctoral Researcher with the Technical University of Berlin and the Max Planck Institute, Germany. From 2012 to 2018 he worked as lead researcher at GE Global Research Europe, Germany. He is currently with Austrian Institute of Technology AIT as research engineer in Vienna, Austria. His research interests cover a wide range of control applications, in particular stability issues in power systems. Dr. Anta received the Fulbright Scholarship in 2005, the Alexander von Humboldt Fellowship in 2011, was a Finalist for the Student Best Paper Award at the IEEE Conference on Decision and Control in 2008, and received the 2010 EMSOFT Best Paper Award and the IEEE CSS George S. Axelby Award in 2011, and won the IEEE PES General Meeting 2020 best paper award.	
\end{IEEEbiography}

\begin{IEEEbiography}[{\includegraphics[width=1in,height=1.25in,clip,keepaspectratio]{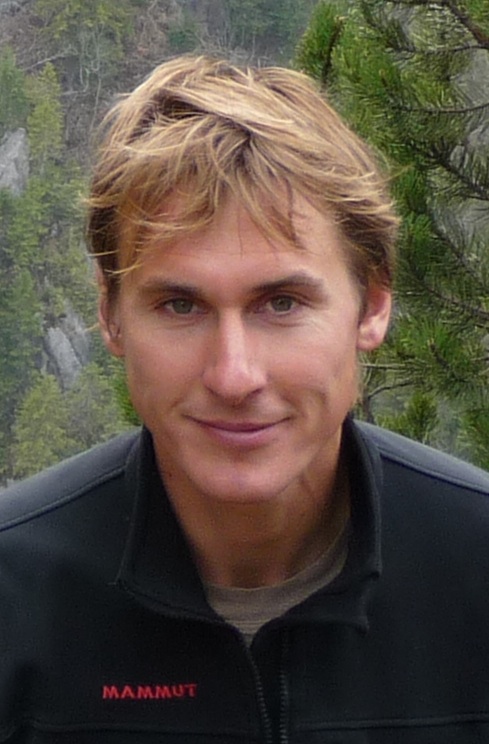}}]{Florian D\"{o}rfler} (S’09–M’13) Florian Dörfler is an Associate Professor at the Automatic Control Laboratory at ETH Zürich. He received his Ph.D. degree in Mechanical Engineering from the University of California at Santa Barbara in 2013, and a Diploma degree in Engineering Cybernetics from the University of Stuttgart in 2008. From 2013 to 2014 he was an Assistant Professor at the University of California Los Angeles. His primary research interests are centered around control, optimization, and system theory with applications in network systems such as electric power grids, robotic coordination, and social networks. He is a recipient of the distinguished young research awards by IFAC (Manfred Thoma Medal 2020) and EUCA (European Control Award 2020). His students were winners or finalists for Best Student Paper awards at the European Control Conference (2013, 2019), the American Control Conference (2016), the PES General Meeting (2020), and the PES PowerTech Conference (2017). He is furthermore a recipient of the 2010 ACC Student Best Paper Award, the 2011 O. Hugo Schuck Best Paper Award, the 2012-2014 Automatica Best Paper Award, the 2016 IEEE Circuits and Systems Guillemin-Cauer Best Paper Award, and the 2015 UCSB ME Best PhD award.	
\end{IEEEbiography}
\end{document}